\documentclass[EJP,preprint]{ejpecp}

\usepackage[T1]{fontenc}
\usepackage[utf8]{inputenc}
\usepackage[authoryear,round]{natbib}
\usepackage{xcolor}
\usepackage{cleveref}

\SHORTTITLE{Ville's Identity}
\TITLE{The Exact Ville Identity with Applications to E-Values}

\AUTHORS{%
Victor H.~de la Pe\~{n}a\footnote{Department of Statistics, Columbia University,
New York, USA.}
\and
Michael J.~Klass\footnote{Departments of Statistics and Mathematics,
University of California, Berkeley, USA.}
\and
Victor K.~de la Pe\~{n}a\footnote{Infremacy, LLC, Fairfield, Connecticut, USA.}
\and
Alexander~Ly\footnote{Machine Learning Group, Centrum Wiskunde \& Informatica,
Amsterdam, the Netherlands.}
\and
Logan~Neishloss\footnote{Loquitor LLC, Levittown, Pennsylvania, USA.}
\and
Fangyuan~Lin\footnotemark[1]}

\AMSSUBJ{60G42}
\AMSSUBJSECONDARY{60G40 ; 62L10}
\SUBMITTED{July 27, 2026}
\ACCEPTED{}
\VOLUME{0}
\YEAR{2026}
\PAPERNUM{0}
\DOI{10.1214/YY-TN}

\ABSTRACT{Let \( (M_{n}, \mathcal{F}_{n})_{n\geq 0} \) be a nonnegative supermartingale with \( M_{0}=1 \), and let \( T_{b} \) be the first time \( M_{n} \) crosses \( b > 1 \). Ville's inequality gives \( P(T_{b} < \infty) \leq 1/b \). We show that, at every finite time horizon \( N \),
\begin{align}
\nonumber
1=b P(T_{b} \leq N) + V_{b,N} + D_{b,N} + R_{b,N},
\end{align}
where \( V_{b, N} \) is the expected overshoot at crossing, \( D_{b, N} \) is the cumulative predictable decrease before crossing, and \( R_{b, N} \) is the expected value of paths that up to the time horizon \( N \) remained below the boundary \( b \). The identity follows from telescoping and can be made horizon free without additional integrability assumption. It gives a formula for the slack in Ville's inequality and thus the conservatism of the threshold test \( M_{n} \geq b \). For likelihood ratio martingales, the same terms also describe power: \( R_{b, N} \) is the probability, under the alternative, of not crossing by \( N \), whereas \( b P (T_{b} \leq N)+V_{b,N} \) is the power.

The identity also allows us to verify horizon-free guarantees from a finite computation: \( P(T_{b} <\infty) \) lies between \( P(T_{b} \leq N) \) and \( P(T_{b} \leq N)+R_{b,N}/b \), so a proposed threshold test \( M_{n} \geq b_{\alpha} \) can be shown to be of level \( \alpha \in (0, 1) \) for every time horizon based on computations at the single horizon \( N \). We illustrate how Ville's identity yields sharper thresholds for Bernoulli and exponential likelihood-ratio tests, and for a conformal test of exchangeability, demonstrating sampling based calibration under a composite, nonparametric null. Lastly, we also establish a continuous-time analogue of Ville's identity.}

\newcommand{\E}{\mathbb E}
\newcommand{\Bernoulli}{\operatorname{Bernoulli}}
\newcommand{\Unif}{\operatorname{Unif}}

\usepackage{booktabs}

\newcommand{\refFig}[1]{Fig.~{\ref{#1}}}
\newcommand{\refEq}[1]{Eq.~{\eqref{#1}}}
\newcommand{\refTab}[1]{Table~{\ref{#1}}}
\newcommand{\refSec}[1]{Section~{\ref{#1}}}

\newcommand{\refProp}[1]{Proposition~{\ref{#1}}}

\newcommand{\bibOrder}[1]{}

\newcommand{\Ef}{\mathbb{E}}

\newcommand{\Fc}{\mathcal{F}}
\newcommand{\Gc}{\mathcal{G}}

\newcommand{\Hc}{\mathcal{H}}

\newcommand{\Pf}{\mathbb{P}}

\newcommand{\der}{\mathrm{d}}

\newcommand{\Exp}{\textnormal{Exp}}

\newcommand{\ind}[1]{\mathbf{1} \left \{ #1 \right \}}

\hypersetup{%
  pdftitle={Ville's Identity for Nonnegative Supermartingales},%
  pdfauthor={Victor H. de la Pena, Michael J. Klass, Victor K. de la Pena, Alexander Ly, Logan Neishloss, and Fangyuan Lin}%
}

\begin{document}

\section{Introduction}
\label{secIntroduction}
Ville's inequality underlies modern so-called safe anytime-valid \( e \)-value-based null hypothesis tests \citep[e.g.,][]{grunwald2024safe,ly2025tutorial,ramdas2023game}. Here, the null hypothesis comprises a collection of measures \( \Hc_{0} \) and time is associated with a filtration \( \Fc \). Ville's inequality holds for \( E \)-processes \( (E_{n}) := E_{1}, E_{2}, \ldots \), which are nonnegative processes that satisfy 
\begin{align}
\label{eqEProcessProperty}
\sup_{P \in \Hc_{0}} \Ef_{P} [ E_{\tau} ] \leq 1
\end{align}
for all stopping times \( \tau \) with respect to \( \Fc \). %
Of particular interest are test supermartingales \( (M_{n}) \), i.e., nonnegative supermartingales with \( M_{0}=1 \) that fulfill condition \refEq{eqEProcessProperty} due to Doob's optional stopping theorem. %

Furthermore, test supermartingales often form the building blocks of \( E \)-processes \citep[e.g.][]{ramdas2020admissible}. %
In fact, we view \( E \)-processes as the encompassing class of processes for which Doob's optional stopping theorem holds. We therefore freely refer to \( E \)-processes and test supermartingales as \emph{evidence} against the null, or the \emph{wealth} of a gambler in an unfavorable game under \( P \). %
A direct application of Markov's inequality to \refEq{eqEProcessProperty} reveals that \( \sup_{P \in \Hc_{0}} P(E_{\tau} \geq b) \leq 1/b \) and the particular choice of the first crossing time
\begin{align}
T_{b} =\inf \{ n \geq 1: E_{n} \geq b \} , 
\end{align}
leads to Ville's inequality \citep{ruf2023composite}: 
%
%
\begin{align}
\label{eqVille}
\sup_{P \in \Hc_{0}} P \! \left(\sup_{n \geq 0} M_{n} \geq b \right) \leq \frac{1}{b} ,  \qquad b>1.
\end{align}
%
%
%
%
%
%
%
Hence, the threshold test that rejects the null hypothesis \( \Hc_{0} \) whenever, or as soon as, \( M_{n} \geq 1/\alpha \) has level \( \alpha \) under \emph{arbitrary monitoring}, since \refEq{eqEProcessProperty} holds for all, possibly unbounded, stopping times. The price of such a general guarantee is that the actual rejection probability can be appreciably smaller than \( \alpha \). This is not automatically a defect, as a sequential test of level 5\% is permitted to ever falsely reject with only, for instance, \( 3.2\% \) under the null. This difference, perhaps due to the use of a conservative test, or a commitment to a finite time horizon \( N \), suggests that the unused size can be traded off for power. Though, exact use of the significance level is not an objective by itself, but usually a consequence of an optimality argument \citep{lehmann-romano2005}. %

We study the trade-off between unused \( \alpha \)-level and power through a finite horizon identity that allows us to define more powerful threshold tests \( M_{n} \geq b \) with less stringent boundaries \( 1 < b < 1/\alpha \). %
Below we prove Ville's identity for test supermartingales \( (M_{n}) \) which states that for each \( P \in \Hc_{0} \) and at every finite time horizon \( N \) the following identity holds: 
\begin{align}
\label{eqVillesIdentity}
1=b P(T_{b} \leq N)+V_{b,N}+D_{b,N}+R_{b,N}.
\end{align}
The first term \( b P(T_{b} \leq N) \) is the boundary level multiplied by the crossing probability, \( V_{b,N} \) is the expected overshoot. The term \( D_{b,N} \) sums the process's cumulative expected decrease, and it is nonzero when \( M_{n} \) is a strict supermartingale. The term \( R_{b,N} \) is the expected value of sample paths that up to the time horizon \( N \) remained below the boundary \( b \). All four terms are nonnegative and are defined with respect to the same measure \( P \). 

Ville's identity shows three causes of finite-horizon conservatism compared to the universal boundary \( b=1/\alpha \):
\begin{align}
\label{eqIntroSlack}
\alpha - P(T_{1/\alpha} \leq N)  =\alpha \bigl ( V_{1/\alpha,N} + D_{1/\alpha,N} + R_{1/\alpha,N} \bigr) . 
\end{align}
If \( M_{n} \) is the likelihood ratio for an alternative \( Q \) against a null \( P \), then there is also a power interpretation:
\begin{align}
\label{eqIntroPower}
Q(T_{b} \leq N)=b P(T_{b} \leq N) + V_{b,N}, \qquad Q(T_{b} > N)=R_{b,N}.
\end{align}
In this case, the residual, calculated under the null, is exactly the type II error under the alternative. It also immediately shows that for any \( b > 1 \) the sequential \( M_{n} \geq b \) test is unbiased with respect to the point alternative \( Q \), because it has a higher rejection probability under the alternative than under the null \citep{lehmann-romano2005}. 

If the terms in Ville's identity are computable under the null, then we can gain power by deriving a sequential test that rejects \( P \) whenever \( M_{n} \geq b_{N, \alpha} \) with \( 1 < b_{N, \alpha} < 1/\alpha \) for all \( n \) up to the time horizon \( N \). %
%
%
This trades off the anytime-valid (but potentially crude) threshold \( 1/\alpha \) for a less stringent boundary \( b_{N, \alpha} < 1/\alpha \) at the price of committing to the horizon \( N \): stopping earlier remains permitted, but monitoring beyond \( N \) might lose type I error control. 

Ville's identity also allows us to derive anytime-valid threshold tests \( M_{n} \geq b_{\alpha} \) without committing to any finite horizon \( N \), where \( b_{N, \alpha} < b_{\alpha} \leq 1/\alpha \). %
%
%
To verify that a candidate boundary \( b_{\alpha} \) indeed never leads to a false rejection probability of more than \( \alpha \), we will show the bracketing inequality: 
%
%
\begin{align}
\label{eqIntroBracket}
P(T_{b} \leq N) \; \leq \; P(T_{b} <\infty) \; \leq \; P(T_{b} \leq N) + \frac{R_{b,N}- R_{b} }{b} \; \leq \; P(T_{b} \leq N)+\frac{R_{b,N}}{b},
\end{align}
where \( R_{b} := \lim_{N \rightarrow \infty} R_{b,N} \). Both the lower and upper bounds are computed at a single finite horizon. As \( N \) increases the lower bound rises and the upper bound falls, and the two meet in the limit whenever \( M_{n} \rightarrow 0 \) under the null. A (candidate) threshold \( b_{\alpha} \leq 1/\alpha \) whose upper bound in \refEq{eqIntroBracket} is at most \( \alpha \) therefore extends to any horizon: a finite computation with a horizon-free conclusion. %
The three threshold tests with boundaries \( b_{N, \alpha}, b_{\alpha} \) and \( 1/\alpha \) are ordered in terms of power. The boundary \( b_{\alpha} \) rejects more often than the universal threshold \( 1/\alpha \) while being valid at every horizon, but it is less powerful compared to the horizon-committed threshold test \( M_{n} \geq b_{N, \alpha} \), which is only valid up to time \( N \). For a composite null, either calibration must hold under every member of the null simultaneously. \refSec{secConformal} illustrates numerical calibration for a conformal test of exchangeability, whose null is composite and nonparametric.

%
The classical background is Ville's inequality, Doob's martingale theory, and Wald's sequential probability/likelihood ratio tests \citep{doob1953stochastic,ville1939etude,wald2004sequential}. Modern applications include \( E \)-processes, anytime-valid confidence sequences, and time-uniform concentration bounds \citep{grunwald2024safe, howard2020time, howard2021time, ly2025tutorial, ramdas2023game, vovk2021values}. A recent application shows how \( E \)-processes enable efficient, adaptive label auditing that reduces expert workload, achieving greater efficiency than the confidence-interval baselines based on concentration inequalities and asymptotic normality \citep{lin2026not}. Overshoot has long been studied through renewal theory and corrected sequential approximations \citep{lorden1970excess, siegmund2013sequential, woodroofe1982nonlinear}; recent work improves \( e \)-value thresholds or redesigns sequential tests to reduce overshoot \citep{blier-wong-wang2026, fischer2026improving, ramdas2026randomized}. Conformal test martingales provide sequential tests of exchangeability \citep{vovk2003testing,vovk2021retrain}; one such test is the composite-null example of \refSec{secConformal}. Closely related work by \citet{ding2026reference} calibrates
finite-horizon rejection boundaries using order statistics of
independent reference-null path maxima. Their exchangeable-rank construction provides finite-sample type I error control jointly over calibration and test randomness, while leaving the evidence
process unchanged. Our decomposition additionally provides lower and upper bounds on the horizon-free crossing probability using finite-horizon quantities. The simulations in \cref{secConformal} illustrate how these bounds can guide threshold calibration; formal guarantees for the simulated thresholds would additionally require accounting for simulation uncertainty. Group-sequential alpha spending methods provide a related model-based approach to allocating type I error across looks \citep{kim1987design, lan1983discrete, wassmer2025group}. Extensions of Ville's inequality to other domains and moving boundaries are complementary to the present accounting \citep{koolen2025generalisation, wang2025extended}.

The proofs below are self-contained. The broader conservation law for extended-valued stopping times in \citet{klass-delapena2026} provides useful context but is not needed for the main result. Conditioning and telescoping do more than make terms disappear: as in Doob's decomposition, they reveal useful structure---here in unused testing level, and in accumulated selection bias in \citet{delapena-lin-delapena2026}, with applications to mitigating the winner's curse in dose-ranging trials \citep{de2026mitigating}. We begin with the finite-horizon identity in \refSec{secFinite}, then extend it to the horizon-free setting in \refSec{secInfinite}. 
\refSec{secTesting}--\refSec{secExamples} develop the statistical consequences and examples. \refSec{secContinuousTime} presents the continuous-time analogue of the identity. \refSec{secDiscussion} concludes. 

\section{The finite-horizon Ville's identity}
\label{secFinite}
Let \( (\Omega, \Fc, P) \) carry a filtration \( (\Fc_{n})_{n \geq 0} \) and we write \( (M_{n}, \Fc_{n}) \) for a nonnegative supermartingale with \( M_{0}=1 \) almost surely. We define its predictable decrease as
\begin{align}
\label{eqDecrement}
\Delta_{n}:=M_{n-1}-\Ef [M_{n} \, | \, \Fc_{n-1}]\geq 0 \quad \text{almost surely}.
\end{align}
Note that it quantifies a change in conditional expectation and not a statistical loss. We fix \( b>1 \), a finite time horizon \( N \), and define
\begin{align}
\label{eqPbn}
p_{b,N} &:= P(T_{b} \leq N), \\
\label{eqVbn}
V_{b,N} &:= \Ef \left[ (M_{T_{b} }-b)  \ind{ T_{b} \leq N } \right], \\
\label{eqDbn}
D_{b,N} &:= \sum_{n=1}^{N} \Ef\left[ \Delta_{n} \ind{ T_{b} \geq n } \right], \\
\label{eqRbn}
R_{b,N} &:=\Ef \left[ M_{N} \ind{ T_{b} > N } \right]. 
\end{align}
The indicator in \refEq{eqDbn} is \( \Fc_{n-1} \)-measurable, and a decrement is counted only while the process has not crossed the boundary \( b \). 

\begin{center}
\small
\begin{tabular}{@{}p{.16\linewidth}p{.39\linewidth}p{.16\linewidth}p{.16\linewidth}@{}}
\toprule
Quantity & Meaning & As $N$ increases & As $b$ increases\\
\midrule
\( b p_{b,N} \) & boundary contribution from paths that cross by \( N \)  & Increases & --\\
\( V_{b,N} \) & (unconditional) expected overshoot at the crossing time & Increases & -- \\
\( D_{b,N} \) & cumulative predictable decrease before crossing & Increases & Increases\\
\( R_{b,N} \) & expected value of paths still below the boundary at \( N \) & Decreases & Increases\\
\bottomrule
\end{tabular}
\end{center}

\begin{theorem}[Finite-horizon Ville's identity]
\label{thmFinite}
For every test supermartingale, every boundary \( b>1 \), and every finite \( N \in \mathbb Z_{> 0} \),
\begin{align}
\label{eqFiniteIdentity}
1=b p_{b,N}+V_{b,N}+D_{b,N}+R_{b,N}.
\end{align}

All four terms on the right-hand side are finite and nonnegative. \( \hfill \diamond \)
\end{theorem}

\begin{proof}
The stopped process has the pathwise expansion by telescoping
\begin{align}
M_{T_{b} \wedge N} = M_{0} + \sum_{n=1}^{N} (M_{n}-M_{n-1}) \ind{ T_{b} \geq n }.
\end{align}
Conditioning on \( \Fc_{n-1} \) and taking expectations gives
\begin{align}
\label{eqStoppedExpectation}
  \Ef [M_{T_{b} \wedge N}] =1-\sum_{n=1}^{N} \Ef [\Delta_{n} \ind{ T_{b} \geq n } ] =1-D_{b,N}.
\end{align}
On the other hand, %
\begin{align}
\label{eqSumBPVbAndRb}
\Ef [M_{T_{b} \wedge N}] &=\Ef [M_{T_{b} } \ind{ T_{b} \leq N }  ]+\Ef [M_{N} \ind{ T_{b} >N } ] \\ 
  &=bP(T_{b} \leq N)+\Ef [(M_{T_{b} }-b) \ind{ T_{b} \leq N }] +R_{b,N}.
\end{align}
Equating the two expressions yields \refEq{eqFiniteIdentity}. Nonnegativity is immediate from the definitions, and the identity then implies finiteness.
\end{proof}

\begin{corollary}[Finite-horizon slack]
\label{corFiniteSlack}
For every \( b>1 \) and \( N<\infty \),
\begin{align}
\label{eq:finite-slack}
  \frac{1}{b}-p_{b,N}=\frac{V_{b,N}+D_{b,N}+R_{b,N}}{b}.
\end{align}
In particular, \( p_{b,N}\leq 1/b \), with equality if and only if all three correction terms vanish. \( \hfill \diamond \)
\end{corollary}

The unconditional definition of \( V_{b,N} \) matters when crossing is impossible. It is always defined, whereas the conditional \( \Ef [M_{T_{b} }-b \, | \,  T_{b} \leq N] \) is not when \( p_{b,N}=0 \).

\begin{example}[All three corrections in one step]
\label{exOneStep}
Let \( M_{0}=1 \) and 
\begin{align}
M_{1}=\begin{cases}
4,&\text{with probability }0.1,\\
1,&\text{with probability }0.3,\\
0,&\text{with probability }0.6,
\end{cases}
\end{align}
and keep the process constant after time \( 1 \). Since \( \Ef [ M_{1}]=0.7 \), this is a strict supermartingale. At \( b=3 \) and \( N=1 \),
\begin{align}
3p_{3,1}=0.3,\qquad V_{3,1}=0.1,\qquad D_{3,1}=0.3,\qquad R_{3,1}=0.3.
\end{align}
Thus \( 1=0.3+0.1+0.3+0.3 \). Overshoot, predictable decrease, and a noncrossing remainder can all be present even in the one-step problem. \( \hfill \diamond \)
\end{example}


\subsection{Subtracting a lower bound}
\label{secLowerBound}
A known positive lower bound on the process \( M \) allows us to sharpen Ville's inequality directly. Suppose that \( \varepsilon \in [0,1] \) is a deterministic constant such that \( M_{n} \geq \varepsilon \) almost surely for every \( n \geq 0 \). Then \( (M_{n} - \varepsilon) \) is a nonnegative supermartingale starting at \( 1-\varepsilon \), and it crosses \( b-\varepsilon \) exactly when \( M_{n} \) crosses \( b \). Applying Ville's inequality to this shifted process gives, for \( \varepsilon > 0, b>1 \),
\begin{align}
\label{eqShiftBound}
P(T_b<\infty) \leq \frac{1-\varepsilon}{b-\varepsilon} < \frac{1}{b} .
\end{align}
For example, if the process starts at \( 1 \) and always remains at least \( 1/2 \), its probability of ever reaching \( 2 \) is at most \( 1/3 \), rather than the usual bound \( 1/2 \).

Ville's identity explains where this improvement comes from. The residual \( R_{b,N}=\Ef[M_{N}\ind{T_b>N}] \) records the expected evidence/wealth remaining on paths that have not crossed the threshold \( b \) by time \( N \). Since each such path remains above \( \varepsilon \),
\begin{align}
R_{b,N}\ge\varepsilon(1-p_{b,N}).
\end{align}
Keeping this contribution in the identity gives
\begin{align}
1 =bp_{b,N}+V_{b,N}+D_{b,N}+R_{b,N} \geq bp_{b,N} + \varepsilon(1-p_{b,N}).
\end{align}
Rearranging and letting \( N \rightarrow \infty \) recovers \refEq{eqShiftBound}. Thus, the shift improves the bound by accounting for wealth that must remain on noncrossing paths.

Note that subtracting \( \varepsilon \) leaves the overshoot and predictable decrease terms unchanged, and the identity becomes 
\begin{align}
\label{eqShiftIdentity}
\begin{split}
1-\varepsilon ={} &(b-\varepsilon) p_{b,N}  +V_{b,N}+D_{b,N} + \Ef[(M_{N}-\varepsilon) \ind{T_b>N}].
\end{split}
\end{align}
The bound in \refEq{eqShiftBound} decreases as \( \varepsilon \) increases. The largest admissible deterministic lower bound is
\begin{align}
\varepsilon_* =\operatorname*{ess\,inf}_{\omega}  \inf_{n\ge0}M_{n}(\omega).
\end{align}
It suffices to bound the stopped process \( M_{n\wedge T_b} \) from below, since values after the first crossing do not affect the argument. This constant-shift argument is a direct consequence of Ville's inequality; more general time-dependent lower bounds and thresholds are considered by \citet{koolen2025generalisation}.

\subsection{The corresponding submartingale identity}
\label{secSubmartingale}
For a supermartingale, predictable decreases reduce the wealth available for crossing a boundary. For a submartingale, the direction is reversed: predictable increases provide additional wealth that can contribute to a crossing. The same stopped telescoping argument gives an identity that accounts for this increase.

Let \( (X_{n})_{n \geq 0} \) be an integrable submartingale, and define
\begin{align}
T_{b}^{X} := \inf \{ n \geq 1  \, : \, X_{n} \geq b \}.
\end{align}
Write
\begin{align}
\Gamma_{n} = \Ef[X_{n}\mid\Fc_{n-1}]-X_{n-1} \geq 0, \qquad A_{b,N}=\sum_{n=1}^{N} \Ef[\Gamma_n\ind{T_{b}^{X}\ge n}].
\end{align}
Here \( A_{b,N} \) is the expected cumulative predictable increase up to \( T_{b}^{X}\wedge N \). Since \( \{T_{b}^{X}\ge n\} \in \Fc_{n-1} \), stopped telescoping gives
\begin{align}
\Ef[X_{T_{b}^{X}\wedge N}] =\Ef[X_0]+A_{b,N}.
\end{align}
Splitting the stopped value according to whether a crossing has occurred therefore yields
\begin{align}
\Ef[X_0]+A_{b,N} ={}&bP(T_{b}^{X} \leq N) +\Ef[(X_{T_{b}^{X}}-b)\ind{T_{b}^{X} \leq N}] +\Ef[X_{N}\ind{T_{b}^{X}>N}].
\end{align}
The terms on the right have the same interpretation as before: the boundary contribution, the expected overshoot, and the expected value remaining on noncrossing paths. The difference is on the left: instead of \( 1-D_{b, N} \), we have \( 1+A_{b,N} \). Crossings are made easier, since the predictable drift helps the process toward the boundary \( b \) rather than pulling it away.

If \( X_{n} \geq 0 \) and \( b>0 \), the last two terms are nonnegative, so
\begin{align}
P\left(T_{b}^{X}\le N\right) \leq \frac{\Ef[X_0]+A_{b,N}}{b}.
\end{align}
Unlike Ville's inequality for a nonnegative supermartingale, this bound includes the predictable increase.



\section{Horizon-free Ville's identity and power via a change of measure}
\label{secInfinite}

%
The finite-horizon identity passes to the limit without an extra assumption on the stopped process. 

\begin{theorem}[Horizon-free Ville's identity]
\label{thmInfinite}
Let
\begin{align}
p_{b}=P(T_{b} <\infty),\quad V_{b}=\Ef [(M_{T_{b} }-b) \ind{ T_{b} <\infty }  ].
\end{align}
Then \( M_{n} \) converges almost surely to a finite limit \( M_{\infty} \), and
\begin{align}
\label{eqResidualLimit}
D_{b}=\lim_{N \to \infty} \sum_{n=1}^{N} \Ef [\Delta_{n} \ind{ T_{b} \geq n }  ] \text{ and } R_{b} := \lim_{N \to \infty} R_{b,N}  =\Ef [M_{\infty} \ind{ T_{b} =\infty } ] . 
\end{align}
Moreover, %
\begin{align}
\label{eqInfiniteIdentity}
1=b p_{b}+V_{b}+D_{b}+R_{b}.
\end{align}
Consequently, \( V_{b}, D_{b}, R_{b} \in [0,1] \), and the horizon-free slack is \( \frac{1}{b}-p_{b}=\frac{V_{b}+D_{b}+R_{b}}{b} \). \( \hfill \diamond \)
\end{theorem}

\begin{proof}
The sequences \( p_{b,N} \), \( V_{b,N} \), and \( D_{b,N} \) are nondecreasing in \( N \). By Theorem~{\ref{thmFinite}} they are bounded and therefore 
converge to \( p_{b} \), \( V_{b} \), and \( D_{b} \). The finite-horizon identity then forces \( R_{b,N} \) to converge and yields \refEq{eqInfiniteIdentity}.

Doob's convergence theorem for nonnegative supermartingales gives \( M_{n} \to M_{\infty} < \infty \) almost surely. Also 
\begin{align}
M_{N} \ind{ T_{b} >N } \longrightarrow M_{\infty} \ind{ T_{b} = \infty } \qquad \text{almost surely}.
\end{align}
On \( \{ T_{b} >N \} \) we have \( M_{n} < b \), so dominated convergence proves \refEq{eqResidualLimit}.
\end{proof}

An immediate consequence is %
\begin{align}
\label{eqStoppedIntegrability}
\Ef [M_{T_{b} } \ind{ T_{b} <\infty } ]=b p_{b}+V_{b} \leq 1.
\end{align}
If \( p_{b} > 0 \), the conditional overshoot
\begin{align}
O_{b}:= \Ef [M_{T_{b} }-b \, | \,  T_{b} <\infty ] = \frac{V_{b}}{p_{b}}
\end{align}
is finite, and \refEq{eqInfiniteIdentity} can be rewritten as
\begin{align}
\label{eqConditionalOvershoot}
p_{b} = \frac{1-D_{b}-R_{b}}{b+O_{b}}.
\end{align}
The additive identity \refEq{eqInfiniteIdentity} remains preferable when \( p_{b}=0 \).

For the usual absorbing case where \( M_{\infty}=0 \) almost surely on \( \{ T_{b} = \infty \} \), we have \( R_{b}=0 \). No uniform-integrability assumption is needed because \( M_{n} \ind{ T_{b} >N } \leq b \). 

\subsection{Likelihood ratio martingales}
\label{secLr}
For a likelihood ratio martingale, the four terms also describe power. 

\begin{proposition}[Change-of-measure interpretation]
\label{propChange}
Let \( P \) and \( Q \) be probability measures on \( (\Omega, \mathcal{F}) \) and \( (\Fc_{n})_{n \geq 0} \) a filtration with \( Q |_{\mathcal{F}_{N}} \ll P |_{\mathcal{F}_{N}} \). Assume that \( (M_{n})_{n \geq 0} \) is the Radon-Nikodym martingale process such that \( M_{n}=\frac{\der Q|_{\Fc_{n}}}{ \der P|_{\Fc_{n}}} \). If the stopping time \( \tau \) is \( P \)-almost surely bounded by \( N \), e.g. \( \tau := T_{b} \wedge N \leq N < \infty \), then for any event \( A \in \Fc_{\tau} \),
\begin{align}
Q(A)=\Ef_{P}\left [ M_{\tau } \ind{A} \right].
\end{align}
In particular, for every finite \( N \)
\begin{align}
\label{eqChangeCross}
Q(T_{b} \leq N) =b p_{b,N}+V_{b,N}, \quad Q(T_{b} >N) =R_{b,N}. 
\end{align}
The horizon-free version is %
\begin{align}
\label{eqChangeInfinite}
Q(T_{b} <\infty)=b p_{b}+V_{b},\qquad Q(T_{b} =\infty)=R_{b}.
\end{align}
Hence, \( R_{b, N} \) and \( R_{b} \) describe the type II error of the \( M_{n} \geq b \) test under \( Q \). \( \hfill \diamond \)
\end{proposition}

\begin{proof}
Because \( \tau \leq N < \infty \), \( A \in \Fc_{\tau} \subset \Fc_{N} \). At the fixed horizon \( N \), we have 
\begin{align}
\Ef_{P} [M_{N} \ind{A} ] = Q(A).
\end{align}
Total probability and Doob's optional sampling theorem give 
\begin{align}
Q(A) = \Ef_{P} [M_{N} \ind{A} ] = \Ef_{P} [ \Ef_{P} [ M_{N} \, | \, \Fc_{\tau}] \ind{A} ] = \Ef_{P}[ M_{\tau} \ind{A} ] . 
\end{align}
\refEq{eqChangeCross} follows from the fact that \( \{ T_{b} > N \} \in \Fc_{N} \), thus, \( Q(T_{b} > N) = R_{b,N} \) and \refEq{eqSumBPVbAndRb}. Letting \( N \rightarrow \infty \) yields \refEq{eqChangeInfinite}. 
\end{proof}

If \( P \) is the null and \( Q \) is a simple alternative, then \( R_{b,N} \) is the type II error of the \( M_{n} \geq b \) test up to the horizon \( N \). 

\section{Sequential testing: Exploiting Ville's identity to create more powerful threshold tests}
\label{secTesting}

\subsection{Test supermartingales and \( E \)-processes}
\label{secEProcesses}
The discussion above can be extended to composite null hypotheses \( \Hc_{0} \). The process \( (M_{n}) \) then has to be an adapted nonnegative process with \( M_{0} = 1 \) that is a supermartingale under every \( P \in \Hc_{0} \). Such a \( \Hc_{0} \) test supermartingale then satisfies the \( E \)-process property \refEq{eqEProcessProperty} and Ville's inequality holds. Hence, the threshold test \( M_{n} \geq 1/\alpha \) is of level \( \alpha \) for all \( P \in \Hc_{0} \) simultaneously. %
%
%

A general \( E \)-process, however, need not have one-step supermartingale dynamics \citep{choe2026combining, ramdas2020admissible, ruf2023composite}. We therefore do not claim that \( D,V,R \) are intrinsic to \( E \)-processes in general. However, when an \( E \)-process is constructed from, or is dominated by, a test supermartingale, the decomposition may be applied to that supermartingale; the resulting terms describe the chosen representation rather than an intrinsic decomposition of the \( E \)-process itself. Under a composite null, the four terms are generally \( P \)-dependent, even when the universal Ville bound holds for every \( P \in H_{0} \). 

\subsection{Using Ville's identity to trade off slackness for power up to a finite horizon}
\label{secArbitraryStopping}
More precise than Ville's inequality, Theorem~\ref{thmFinite} shows that with \( b = 1/\alpha \) the crossing probability by time horizon \( N \) under \( P \) is %
\begin{align}
\label{eqDefaultCalibration}
P(T_{b} \leq N) = \alpha \bigl (1-V_{b, N} - D_{b, N} - R_{b, N} \bigr ) . 
\end{align}
This identity suggests a potentially more powerful test with boundary \( b_{N, \alpha} < 1/\alpha \) in case the sum \( V_{1/\alpha, N} + D_{1/\alpha, N} + R_{1/\alpha, N}  \) is non-zero. 
%

\begin{proposition}[Horizon-specific boundary calibration]
\label{propCalibration}
If \( c_{N, \alpha} \in (1, 1/\alpha] \) satisfies
\begin{align}
\label{eqCalibrationCondition}
P(T_{c_{N,\alpha}} \leq N) \leq \alpha,
\end{align}
then the threshold test that rejects \( P \) upon crossing \( c_{N,\alpha} \) by time \( N \) has level \( \alpha \). If \( c_{N,\alpha} < 1/\alpha \), then a rejection event based on the threshold \( 1/\alpha \) was already rejected at \( c_{N, \alpha} \). Thus, the threshold test \( M_{n} \geq c_{N,\alpha} \) is at least as powerful as the \( M_{n} \geq 1/\alpha \) test. \( \hfill \diamond \)
\end{proposition}

\begin{proof}
The first statement follows from the discussion surrounding \refEq{eqDefaultCalibration} and the second follows from 
\begin{align}
\left \{ \max_{n\leq N}M_{n} \geq 1/\alpha \right \} \subseteq \left \{ \max_{n \leq N} M_{n} \geq c_{N,\alpha} \right \}.
\end{align}
\end{proof}
On a lattice, there might not be a smallest \( c_{N, \alpha} \) such that \refEq{eqCalibrationCondition} holds exactly. In such case, we choose a convenient boundary \( b_{N, \alpha} \) close to the infimum of the feasible set \( \{ c \in (1, 1/\alpha]  \, : \, P(T_{c} \leq N) \leq \alpha \} \) such that the error does not exceed \( \alpha \). Such a threshold \( b_{N, \alpha} \) only guarantees type I error control up to the horizon: A boundary calibrated through \( N \) may exceed level \( \alpha \) if monitoring continues, since \( b_{N, \alpha} \) is nondecreasing in \( N \). Furthermore, the feasible set \( \{ c \in (1, 1/\alpha]  \, : \, P(T_{c} \leq N) \leq \alpha \} \) shrinks as \( N \) grows. 
%
%
%
%

For a composite null, the condition becomes
\begin{align}
\label{eqCompositeCalibration}
\sup_{P \in H_{0}} P(T_{b} \leq N) \leq \alpha.
\end{align}
Calibrating \( b \) to one convenient member of \( H_{0} \) is not enough. The gain from using a threshold test with a lower boundary \( b_{N, \alpha}  \) than the universal boundary \( 1/\alpha \) can shrink or disappear when the least-favourable null is taken into account.

The curve 
\begin{align}
\label{eqSpendingCurve}
A_{b}(n) = P(T_{b} \leq n)
\end{align}
is the cumulative type I error induced by the chosen process and boundary. It may be viewed as a process-specific spending curve. Classical alpha spending methods reverse the direction: they specify a desired cumulative error curve and solve for a sequence of boundaries under a model \citep{brannath2026anytime}. 

\subsection{Horizon-free calibration}
\label{secInfCalibration}
%
Let \( b \) be a postulated boundary such that \( b_{N, \alpha} \leq b \leq 1/\alpha \), then the finite-horizon Ville's identity can be exploited to verify the horizon-free claim \( P(T_{b} < \infty) \leq \alpha \). %
%
%

\begin{proposition}[The horizon-free bracket]
\label{propBracket}
For every test supermartingale \( M_{n} \), every \( b > 1 \), and every finite \( N \),
\begin{align}
\label{eqBracket}
p_{b,N} \; \leq \; p_{b} \; \leq \; p_{b, N} + \frac{R_{b, N}- R_{b}}{b} \; \leq \; p_{b, N} + \frac{R_{b, N}}{b} . 
%
\end{align}
The lower bound is nondecreasing and the upper bounds are nonincreasing in \( N \). \( \hfill \diamond \)
\end{proposition}

\begin{proof}
The first inequality follows from the fact that \( p_{b,N} \) is nondecreasing in \( N \). For the upper bound, we use monotonicity of \( V_{b,N} \uparrow V_{b} \), \( D_{b,N} \uparrow D_{b} \), and therefore \( R_{b, N} \downarrow R_{b} \geq 0 \). By Theorem~\ref{thmInfinite} and Theorem~\ref{thmFinite} we get 
\begin{align}
\nonumber
b \, p_{b}=1-V_{b}-D_{b}-R_{b} \leq 1-V_{b,N}-D_{b,N} - R_{b} = b \, p_{b,N} + R_{b,N} - R_{b} \leq b p_{b,N} + R_{b,N}
\end{align}
Dividing by \( b \) yields \refEq{eqBracket}. 
\end{proof}

The rough width of the bracket is \( R_{b,N}/b \). For a likelihood ratio martingale, Proposition~\ref{propChange} gives \( R_{b,N}=Q(T_{b} >N) \), so the bracket closes automatically for sequential threshold tests that reject with power one under the alternative \citep{ram2026power}. The tighter width involving \( R_{b,N} - R_{b} \) is useful when \( R_{b}>0 \), see also \refSec{secLowerBound}. %
%

In many cases we cannot compute \( V_{b,N}, R_{b,N} \) and \( D_{b,N} \) explicitly, but we are able to sample under the null, say, \( m \) sample paths up to time horizon \( N \). A simple root-finding approach using the bisection method can then be applied to the \( m \) sample paths to determine \( b_{N,\alpha},  b_{\alpha} \in (1, 1/\alpha] \) for any given \( \alpha \in (0, 1) \). When \( D_{b,N} \) is known and for \( b \) a candidate boundary, we can evaluate the empirical versions \( \hat{V}_{b,N}, \hat{R}_{b,N} \) to derive \( \hat{p}_{b,N} \). %
%
%
%
%
%
\refSec{secExamples} illustrates this sampling procedure. The bounds above hold for the exact expectations. Replacing these expectations by sampling averages yields estimated bounds, which need not retain their bounding property in a particular simulation run. We use this procedure for numerical calibration; unless verified independently, the resulting thresholds are estimates, and formal level-\(\alpha\) guarantees require additional control of simulation error. 
%
%

If for the candidate \( b \), the upper bound of \refEq{eqBracket} is at most \( \alpha \), then \( p_{b} \leq \alpha \): rejection upon crossing \( b \) has level \( \alpha \) for any finite stopping time. 
%
On the other hand, if the lower bound exceeds \( \alpha \), then we know that a postulated threshold test \( M_{n} \geq b \) is not anytime-valid at level \( \alpha \), thus, also not suitable for type I error up to \( N \), hence, \( b_{N, \alpha} > b \). To maximize power for a horizon-free test, we let \( b_{\alpha} \) be the smallest \( b \) such that \( p_{b,N} + R_{b,N}/b \leq \alpha \). %
%
%
For a composite null, a boundary whose upper bound is at most \( \alpha \) for every \( P \in H_{0} \) verifies the supremum condition \refEq{eqCompositeCalibration} at every horizon, from finite-horizon computations.


\subsection{The raw \( e \)-value is not changed by the boundary}
\label{secRawEValue}
Choosing a lower decision boundary changes the rejection rule, not the evidence process \citep{chugg2026values}. If a test supermartingale is stopped at a bounded stopping time \( \tau \)---in particular, at a time no later than the planned horizon \( N \)---then \( M_{\tau} \) remains an \( E \)-variable. The unmodified value may therefore be reported and, used in later evidence combination subject to the assumptions of the chosen merging function or meta-analysis protocol \citep{grunwald2024safe, terschure-grunwald2025, vovk2021values, wang2025}.

This is different from replacing \( M_{\tau} \) by a boosted quantity such as \( c M_{\tau} \) with \( c > 1 \) \citep[e.g.,][]{blier-wong-wang2026, fischer2026improving}, or by a transformation that depends on reaching a particular horizon \citep{baas2026adaptive, clerico2026time}. %
%
Such transformation makes the interpretation of evidence depend on both the boundary and the chosen time horizon. Calibrating by only lowering the decision boundary leaves the original \( M_{\tau} \) untouched for later evidence combination. 

\subsection{The reciprocal running maximum}
\label{secPValue}
We can convert a test supermartingale into a conservative \( p \)-value. Let 
%
%
\begin{align}
M^{*} := \sup_{n \geq 0} M_{n}, \text{ and } P_{V} := \frac{1}{M^{*}} .
\end{align}
Since \( M_{0}=1 \), \( P_{V} \in [0,1] \). By Ville's inequality \( P ( M^{*} \geq 1/u) = P (P_{V} \leq u ) \leq u \) yielding the sought-after conservative \( p \)-value. Its distribution is given by \refEq{eqPValueCdf}. 
%

\begin{corollary}[The distribution of \( P_{V} \)]
\label{corPValue}
For \( u \in (0,1) \) and \( b:= 1/u \)
\begin{align}
\label{eqPTauCdf}
P( T_{b} < \infty ) =u \bigl (1-V_{b}-D_{b}-R_{b} \bigr ) ,
\end{align}
so the slack of Ville's bound at the boundary \( b \) is 
\begin{align}
\label{eqPValueSlack}
u-P(T_{b} < \infty )=u \bigl (V_{b} + D_{b} + R_{b} \bigr ).
\end{align}
The event \( \{ P_{V} \leq u \} = \{ M^{*} \geq b \} \) contains the crossing event, that is, 
\begin{align}
\label{eqSetEnclosure}
\{ M^{*} > b \} \subset \{ T_{b} < \infty \} \subset \{ M^{*} \geq b \}
\end{align}
and its probability is the left limit of \refEq{eqPTauCdf}
\begin{align}
\label{eqPValueCdf}
P(P_{V} \leq u ) = \lim_{b' \uparrow 1/u}  \frac{1 - V_{b'} - D_{b'} - R_{b'}}{b'} , 
\end{align}
which might differ from the right limit as shown in Remark~{\ref{remSupremumNoCrossing}}. \( \hfill \diamond \)
\end{corollary}

\begin{proof}
If \( M_{\tau} \geq b \) at a finite \( \tau \), then \( T_{b} \leq \tau \); this gives both inclusions in \refEq{eqSetEnclosure}. Substitute \( b \) into \refEq{eqInfiniteIdentity} gives \refEq{eqPTauCdf}. Finally, \( \{ M^{*} \geq b \} = \cap_{b' < b} \{ T_{b'} < \infty \} \), a decreasing intersection as \( b' \uparrow b \), so continuous from above and \refEq{eqPTauCdf} applied at each \( b' < b \) gives \refEq{eqPValueCdf}. 
%
\end{proof}

\begin{remark}[The supremum can reach the boundary without crossing] 
\label{remSupremumNoCrossing}
The second inclusion in \refEq{eqSetEnclosure} can be strict, so \refEq{eqPTauCdf} cannot in general be read as the distribution function of \( P_{V} \) without the limit \refEq{eqPValueCdf}. Let \( \xi_{1}, \xi_{2}, \ldots \) be independent Bernoulli variables with \( P(\xi_{i} = 1) = b^{-2^{-i}} \). Let \( A := \{ \xi_{i} = 1 \text{ for all } i \} \), so that \( P(A) = 1/b \), %
and consider the Doob martingale \( M_{n} := b P(A \, | \, \xi_{1}, \ldots, \xi_{n})  \). Then \( M_{0} = 1 \); off \( A \) the martingale is absorbed at zero, whereas on \( A \) the martingale equals \( b^{1-2^{-n}} \) and increases to \( b \) strictly from below. %
Hence, \( P(M^{*} \geq b) = 1/b \), whereas \( P(T_{b} < \infty ) = 0 \): The identity holds with \( p_{b} = V_{b} = D_{b} = 0 \) and \( R_{b} = 1 \), and \refEq{eqPValueCdf} returns \(\lim_{b'\uparrow b}P(T_{b'}<\infty)=1/b\), as it should. \( \hfill \diamond \)
\end{remark}

Recall that for test supermartingales Doob's optional stopping theorem holds for all stopping times and this implies \( P(M_{\tau} \geq b) \leq 1/b \), even if \( \tau \) is unbounded. The telescoping technique used to derive finite-horizon Ville's identity and its horizon-free extension, however, restricts itself to finite stopping times. 

Analogously to Corollary~{\ref{corPValue}}, we have for any finite stopping time \( \tau \) 
\begin{align}
\label{eqArbitraryStopping}
\{ M_{\tau} \geq b,\ \tau < \infty \} \subseteq \{ T_b < \infty \} = \bigcup_{n \geq 0 } \{ M_{n} \geq b \} \subseteq \Big \{ \sup_{n \geq 0} M_{n} \geq b \Big \} . 
\end{align}
%
If \( \tau \leq N \) almost surely, then the first event is a subset of the finite-horizon event \( \{ T_{b} \leq N \} \). 
\refEq{eqArbitraryStopping} implies 
\begin{align}
\label{eqAnytimeValidFinite}
P( M_{\tau} \geq b,\ \tau < \infty ) \leq P(T_b < \infty ) = \frac{1 - V_{b} - D_{b} - R_{b}}{b} . 
\end{align}
Remark~{\ref{remSupremumNoCrossing}} shows that finiteness \( \tau \) is necessary, as taking \( \tau = T_{b} \) with \( b=2 \) shows that \( P(M_{\infty} \geq b) = 1/2 \), whereas \( P(T_{b} < \infty) = 0 \). \\
%

The discussion extends naturally to continuous-time nonnegative local martingale \( M_{t} \), see \refSec{secContinuousTime} for the technical details. If \(M\) is a nonnegative local martingale with no positive jumps and \(M_t\to0\) almost surely, then \(V_b=D_b=R_b=0\), and Doob's maximal identity gives
\begin{align}
P \! \left ( \sup_{t \geq 0} M_{t}  \geq b \right ) = \frac{1}{b},
\end{align}
so, in this case, \( P_{V} \) is exactly uniform under the standard max-continuity conditions \citep{nikeghbali-yor2006}. The discrete identity identifies the precise departure from that benchmark.

There is an analogous finite-horizon statement. Let the running maximum and the reciprocal running maximum be 
\begin{align}
M_{N}^{*} := \max_{0 \leq n \leq N} M_{n}, \text{ and } P_{V, N} := \frac{1}{M_{N}^{*}}, 
\end{align}
respectively. The event \( \{ P_{V,N} \leq u \} \) equals \( \{ T_{1/u} \leq N \} \), and Theorem~\ref{thmFinite} gives
\begin{align}
\label{eqPValueFinite}
P(P_{V, N} \leq u) = u \bigl (1-V_{1/u, N}-D_{1/u, N}-R_{1/u, N} \bigr ).
\end{align}
%


\section{Examples}
\label{secExamples}

\begin{figure}[h]
\centering
\includegraphics[width=0.95 \textwidth]{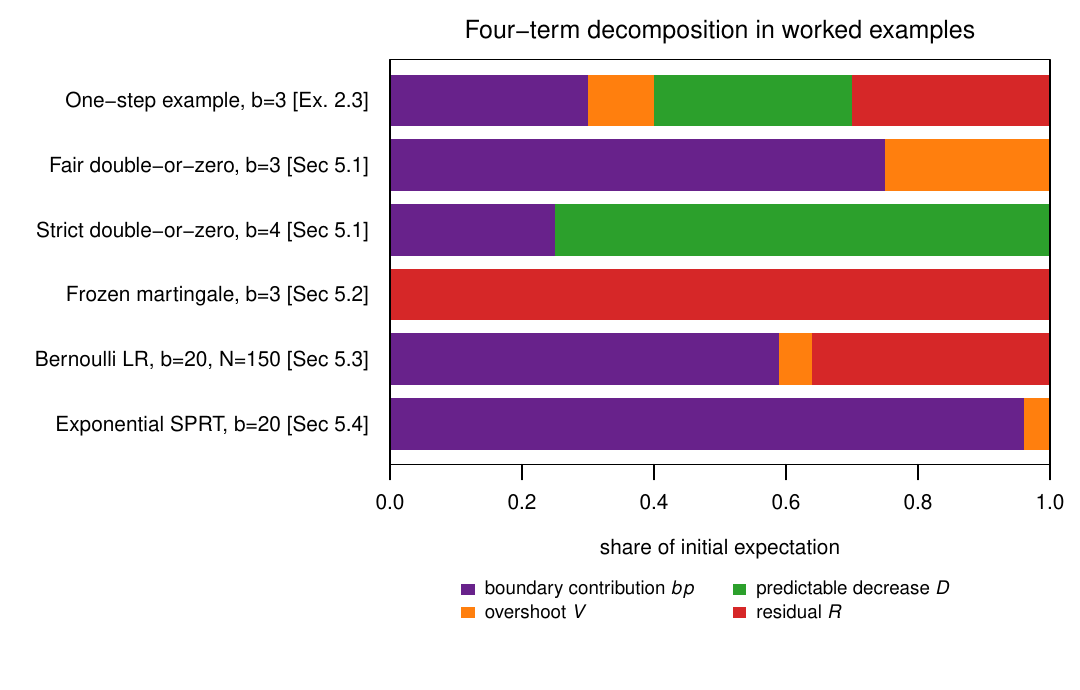}
\caption{Four-term decompositions for the various examples. Each row sums to one. The one-step example has all three corrections; the next three rows isolate overshoot, predictable decrease, and residual. The Bernoulli example is based on the likelihood ratio \( p_{1} = 0.6\) against \( p_{0} = 0.5\).}
\label{figFourTerm}
\end{figure}

\subsection{Double-or-zero processes}
\label{secDoubleZero}
Fix \( q \in (0,1/2] \). Starting from \( M_{0}=1 \), each nonzero step doubles with probability \( q \) and is sent to zero with probability \( 1-q \). Then
\begin{align}
\Ef [M_{n} \, | \, \Fc_{n-1}]=2qM_{n-1},\qquad \Delta_{n}=(1-2q)M_{n-1}.
\end{align}
Let \( k = \lceil \log_{2} b \rceil \). Crossing requires \( k \) consecutive doublings and lands at \( 2^{k} \), so 
\begin{align}
\label{eqDoubleP}
p_{b} &= q^{k}, \\
\label{eqDoubleV}
V_{b}&=(2^{k}-b)q^{k} ,  \\
\label{eqDoubleD}
D_{b}&=1-(2q)^{k}, \\
\label{eqDoubleR}
R_{b}&=0. 
\end{align}
When \( q=1/2 \) and \( b=3 \), the process is a martingale, \( p_{b}=1/4 \), and \( V_{b}=1/4 \); the slack is entirely overshoot. When \( q=1/4 \) and \( b=4 \), \( p_{b}=1/16 \) and \( D_{b}=3/4 \); the slack is entirely predictable decrease. 


\subsection{A pure residual process}
\label{secPureResidual}
Let \( X \sim \Unif(0,2) \), take \( \Fc_{0} \) trivial and \( \Fc_{n}=\sigma(X) \) for \( n \geq 1 \), and define \( M_{0}=1 \), \( M_{n}=X \) for \( n \geq 1 \). This is a bounded martingale. For \( b=3 \), crossing is impossible, so
\begin{align}
p_{b}=V_{b}=D_{b}=0, \qquad R_{b}=\Ef[ X]=1.
\end{align}
This is a nontrivial example of paths that remain below \( b \) without decaying to zero. 


\subsection{A Bernoulli likelihood ratio process at a planned horizon}
\label{secBernoulli}
Let
\begin{align}
H_{0}: X_{i} \sim \Bernoulli(p_{0}),\qquad H_{1}: X_{i} \sim \Bernoulli(p_{1}),
\end{align}
Writing \( S_{n}=\sum_{i=1}^{n} X_{i} \), the likelihood ratio martingale is
\begin{align}
\label{eqBernoulliLr}
M_{n} = \left (\frac{p_{1}}{p_{0}} \right )^{S_{n}} \left (\frac{1-p_{1}}{1-p_{0}} \right )^{n-S_{n}}.
\end{align}
Provided with a finite horizon \( N \) and the fact that the process can only go up and down, we can track the probability distribution of \( M_{n} \) across a binomial tree.%
\footnote{See the Github repository \url{https://github.com/AlexanderLyNL/villesIdentity/} for the R code.} %
For time horizon \( N=50 \), \( p_{0} = 0.5 \) and \( p_{1} = 0.9 \), this yields %
%
%
\begin{align}
\label{eqBernoulliDecomposition}
1= \underbrace{0.69997}_{20p_{20,50}} +\underbrace{0.29964}_{V_{20,50}} +\underbrace{0}_{D_{20,50}} +\underbrace{0.00039}_{R_{20,50}}, \text{ thus, } P(T_{20} \leq 50)=0.035.
\end{align}
By Proposition~\ref{propChange}, the power under \( p_{1} \) is
\begin{align}
Q(T_{20} \leq 50) =20p_{20,50}+V_{20,50}=0.9996, \text{ and } Q(T_{20} > 50) = R_{b,N} = 0.0004
\end{align}
is exactly the probability of no rejection by time \( 50 \) under the alternative with \( p_{1} = 0.9 \). To convert the slackness into power, we use the bisection method applied to the binomial tree and see that for \( \alpha = 0.05 \) and \( N=50 \) we require \( b_{N, \alpha}=16.658372 \), for which
\begin{align}
P(T_{b_{N,\alpha}} \leq 50)=0.04881, \text{ and } Q(T_{b_{N,\alpha}} \leq 50) = 0.9997.
\end{align}
%
%
The threshold \( b_{N, \alpha}=16.658372 \) is correct up to rounding error, as it is based on propagating the distribution across a binomial tree. To evaluate the proposed sampling procedure inspired by \refProp{propBracket} we generated \( m=2,000,000 \) sample paths up to time horizon \( N=50 \) under the null. \refFig{figBernoulli} shows in red the naive estimator \( \tilde{p}_{b,N} := \tfrac{1}{m} \sum_{j=1}^{m} \ind{ \max_{n \in \{ 1, \ldots, N \}} M_{n}^{j}  \geq b} \), where \( m \) is the number of simulated paths, and in yellow the estimator of \( \hat{p}_{b,N} \) based on Ville's identity. 
\begin{figure}[h]
\centering
\includegraphics[width=0.94 \textwidth]{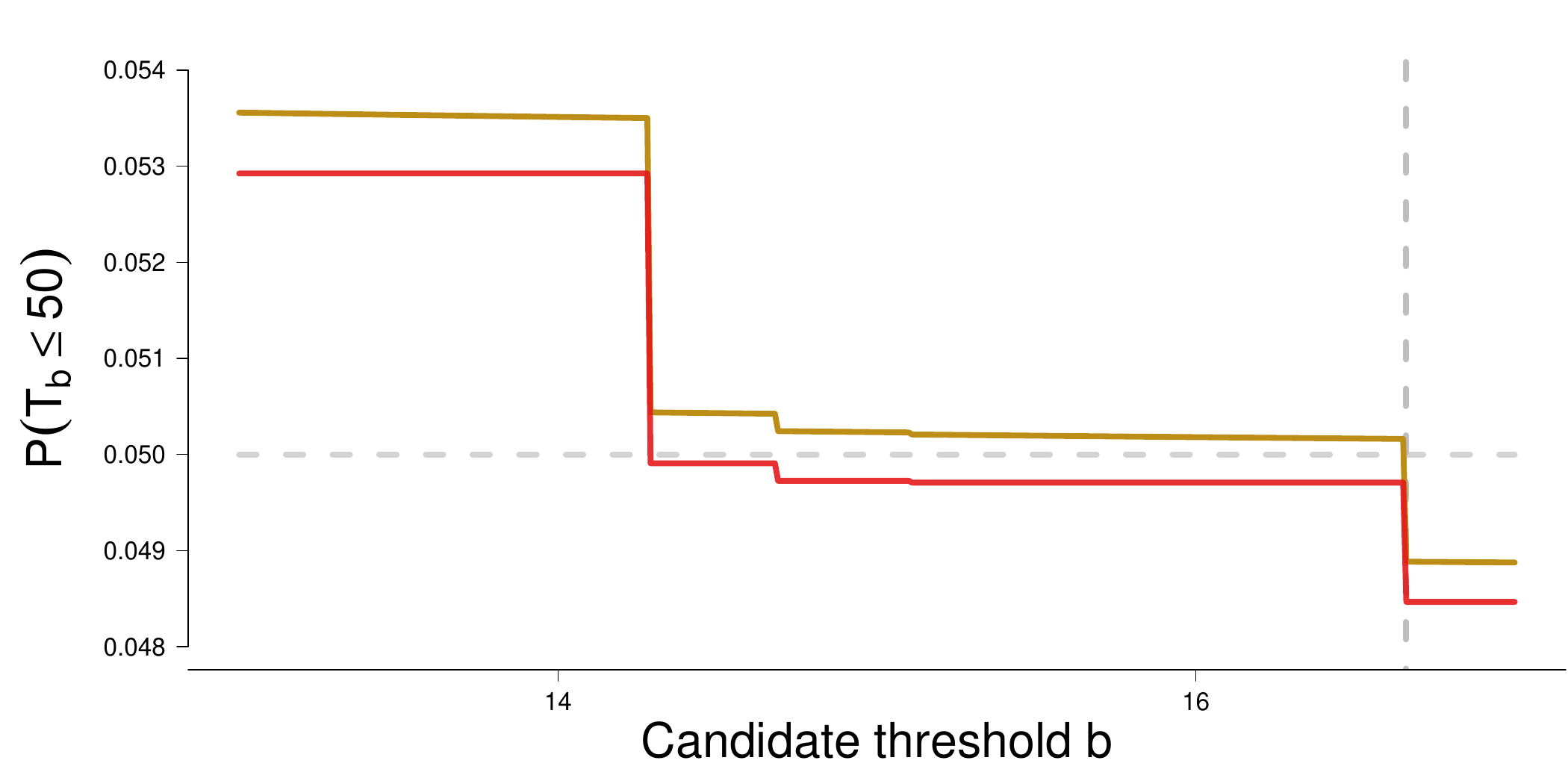}
\caption{When sampling up to a time horizon of \( N=50 \), the threshold can be lowered to \( b_{N, \alpha}=16.658372 \), vertical broken grey line. The naive estimator (red curve) yields an invalid threshold \( \tilde{b}_{N,\alpha} =14.29 < b_{N, \alpha} \), whereas the Ville estimator (yellow curve) leads to a valid threshold \( \hat{b}_{N, \alpha} = 16.66 > b_{N, \alpha}  \).} 
\label{figBernoulli}
\end{figure}
%
The default boundary is valid without a finite time horizon and without computing the exact crossing law. The lower boundary uses more of the five-percent allowance and improves power, but only up to the stated horizon. Note that this approach differs substantially from the one proposed by \citet{baas2026adaptive}, where dynamic programming is used to define the \( E \)-process itself, whereas we focus on lowering the evidence boundary for a given \( E \)-process. %
\begin{table}[htbp]
\centering
\caption{For the Bernoulli likelihood ratio process, with \( p_{1} = 0.9 \) and \( p_{0} = 0.5 \), the gain in power (under \( p_{1} = 0.9 \)) by lowering the universal threshold \( 20 \) to the horizon committed \( b_{N,\alpha} = 16.65838 \) is relatively small.}
\label{tabBernoulli}
\begin{tabular}{@{}rrrrrrr}
\hline
\( b \) & null size & \( bp_{b, N} \) & \( V_{b, N} \) & \( R_{b, N} \) & power \\ 
\hline
20 & 0.0350 & 0.7000 & 0.2996 & 0.0004 & 0.9996 \\ 
16.65838 & 0.0488 & 0.8131 & 0.1865 & 0.0003 & 0.9997 \\ 
\hline
\end{tabular}

\caption{For the Bernoulli likelihood ratio process, with \( p_{1} = 0.6 \) and \( p_{0} = 0.5 \), the gain in power (under \( p_{1} = 0.6 \)) by lowering the universal threshold \( 20 \) to the horizon committed \( b_{N,\alpha} = 13.15491 \) is about 7.3\% points at \( n = 150\). 
}
\label{tabBernoulli06}
\begin{tabular}{@{}rrrrrrr}
\hline
\( b \) & null size & \( bp_{b, N} \) & \( V_{b, N} \) & \( R_{b, N} \) & power \\ 
\hline
20 & 0.0295 & 0.5891 & 0.0504 & 0.3605 & 0.6395 \\ 
13.15491 & 0.0499 & 0.6570 & 0.0557 & 0.2873 & 0.7127 \\ 
\hline
\end{tabular}
\end{table}

It is worth noting that the example is set up to illustrate the difference between the naive estimator and the Ville's identity estimator. 

%
%
%
%
\begin{align}
\hat{p}_{b,N} = \tilde{p}_{b,N} + \frac{ 1 - \hat{M}_{T_{b} \wedge N} }{b}, \text{ where } \hat{M}_{T_{b} \wedge N} := \frac{1}{m} \sum_{j=1}^{m} M_{T_{b} \wedge N}^{j} . 
\end{align}
%

In this simulation, the naive estimator yields a lower (and invalid) threshold
than the estimator based on Ville's identity. The difference
between the estimators is determined by the empirical mean
of the stopped process. 
Although most paths carry little wealth at the stopping horizon, the mean remains one because crossing paths carry larger values.
Underrepresenting these paths in a simulation lowers the empirical
stopped mean and therefore makes \(\hat p_{b,N}\) larger than
\(\tilde p_{b,N}\). The reverse ordering is also possible;
neither estimator is automatically conservative.


If no simulated path crosses \(b\), then
\(\tilde p_{b,N}=0\) and
\(\hat p_{b,N}=(1-\hat R_{b,N})/b\), which is approximately
\(1/b\) when the empirical residual is small. Furthermore, the gain in power in this example is modest. If we used \( p_{1} = 0.6 \) and \( p_{0}= 0.5 \) instead, then the gain in power by using \( \hat{b}_{N, \alpha} \) compared to \( 1/\alpha \) is substantial, see \refTab{tabBernoulli06} and \refFig{figCumulative}. However, the difference between \( \tilde{b}_{N,\alpha} \) and \( \hat{b}_{N, \alpha} \) is then hardly visible.  

\begin{figure}[ht]
\begin{tabular}{cc}
    \begin{minipage}{.48 \textwidth}
	    \centering
	    \includegraphics[width= \linewidth]{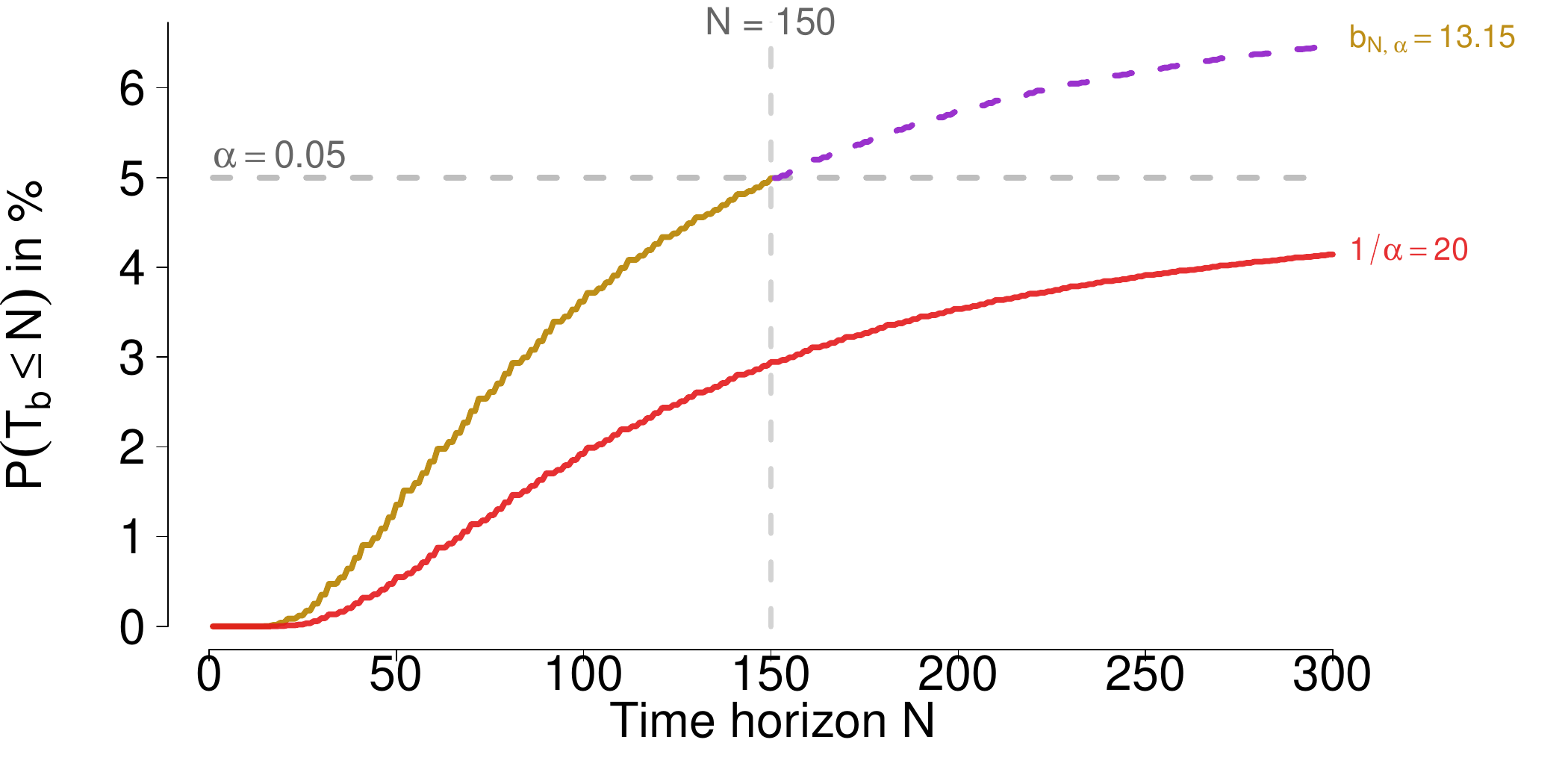}
    \end{minipage}  & %
    \begin{minipage}{.48 \textwidth}
	    \centering
	    \includegraphics[width=\linewidth]{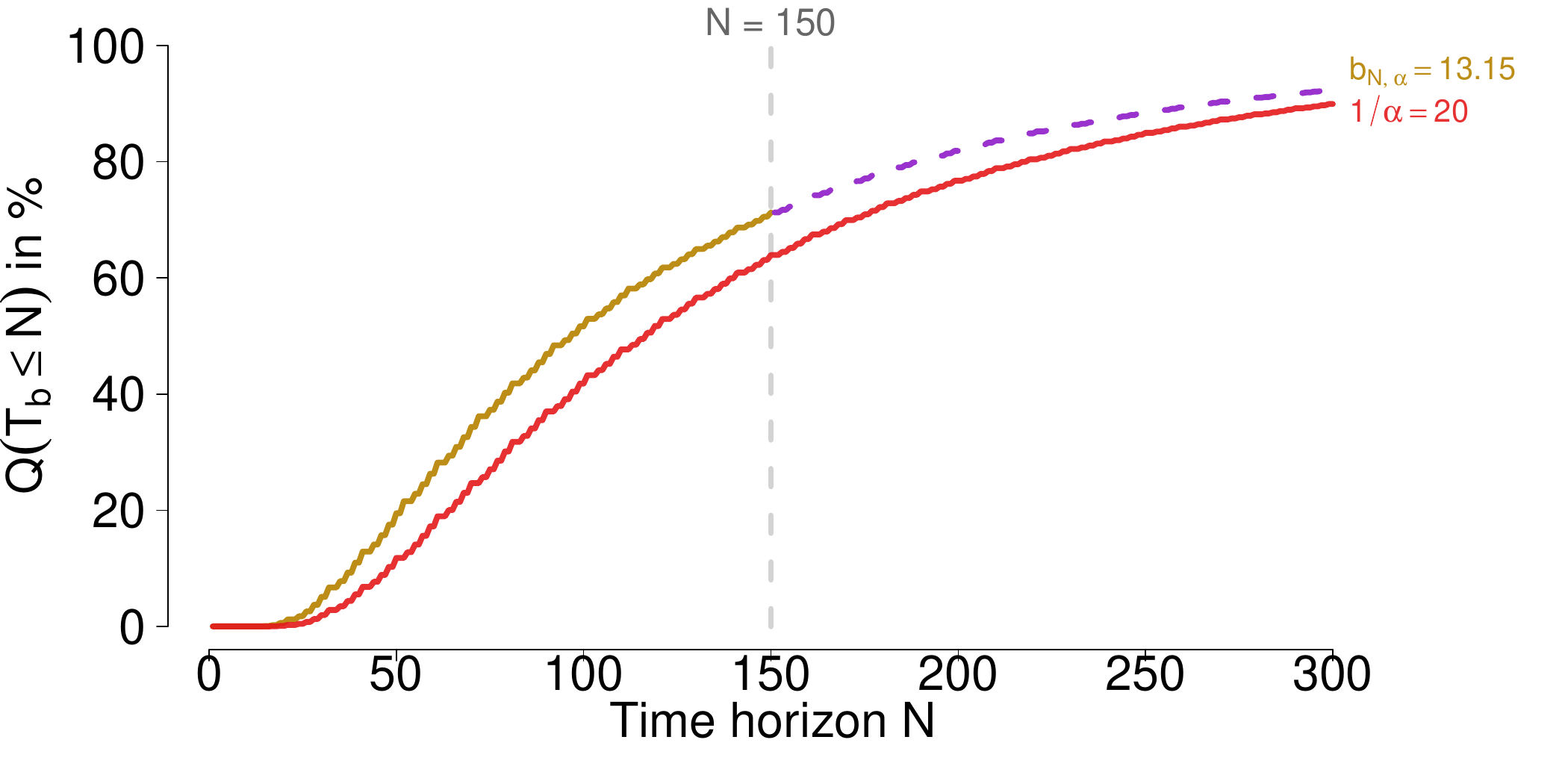}
    \end{minipage}  %
\end{tabular}
\caption{
The yellow and red curves represent the type I error and power curves of the procedure that rejects the null as soon as \( M_{n} \geq b_{N,\alpha}=13.15491 \) and the universal \( M_{n} \geq 1/\alpha=20 \) test, respectively, where \( M_{n} \) defines the \( p_{1}=0.6 \) over the \( p_{0} = 0.5 \) likelihood ratio process. The left panel shows that horizon committed test, yellow curve, loses type I error guarantee beyond the time horizon \( N = 150 \), from which point onwards it is depicted in purple. The universal threshold test (red) remains valid throughout, but has less power at the horizon \( N = 150 \) (right panel).} %
\label{figCumulative}
\end{figure}
Our simulations illustrate that calibration based on Ville's identity can select a valid threshold when the empirical crossing frequency selects an invalid one. Determining when this advantage persists across repeated simulations remains a question for further study.

\subsection{An exponential sequential probability-ratio test}
\label{secExponential}
Let \( X_{1}, X_{2}, \ldots \) be i.i.d. exponential random variables with rate \( \theta \). Test
\begin{align}
H_{0} : X_{i} \sim \Exp(\theta_{0}) \quad \text{against} \quad H_{1}: X_{i} \sim \Exp(\theta_{1}), \qquad 0 < \theta_{1} < \theta_{0} . 
\end{align}
The likelihood ratio 
\begin{align}
\label{eqExponentialLr}
M_{n}=\prod_{i=1}^{n} \frac{\theta_{1} e^{-\theta_{1} X_{i}}}{\theta_{0} e^{-\theta_{0} X_{i}}}
\end{align}
is a martingale under \( H_{0} \). Rewrite the overshoot as \( M_{T_{b}} - b = b ( \tfrac{M_{T_{b}}}{b} - 1) =: b ( e^{\tilde{V}} - 1) \), where given that \( T_{b} = k < \infty \), the log overshoot is 
\begin{align}
\tilde{V}:= \log M_{k} - \log (b) = \log M_{k - 1} - \log(b) + Y_{k} \geq 0 , 
\end{align}
where the log increment \( Y_{k} \) is 
\begin{align}
Y_{k} := \log(\theta_{1} / \theta_{0})+(\theta_{0}-\theta_{1}) X_{k}.
\end{align}
%
On \(\{T_b\geq k\}\), put
\[
C_k=
\frac{\log b-\log M_{k-1}-\log(\theta_1/\theta_0)}
     {\theta_0-\theta_1}>0.
\]
Then
\[
\{T_b=k\}=\{T_b\geq k\}\cap\{X_k\geq C_k\}.
\]
Under \(P\), conditional on the past and crossing at step \(k\),
memorylessness gives \(X_k-C_k\sim\Exp(\theta_0)\).
Consequently, the log overshoot has distribution
\(\Exp(\theta_0/(\theta_0-\theta_1))\), independently of
the crossing time and the pre-crossing history.
A direct computation shows 
\begin{align}
\label{eqExponentialOvershoot}
O_{b} = \Ef [ M_{k} - b \, | \, T_{b}= k ] =  b \Ef [ e^{\tilde{V}} - 1\, | \, T_{b} =k ]  = b \frac{ \theta_{0} - \theta_{1} }{\theta_{1}}. 
\end{align}
Under \(P\), the strong law gives
\[
n^{-1}\log M_n\longrightarrow
\log(\theta_1/\theta_0)+1-\theta_1/\theta_0<0,
\]
so \(M_n\to0\) almost surely and \(R_b=0\). Since \( D_{b}=R_{b}=0 \) and we have, by Ville's identity in conditional overshoot form,
%
%
%
%
\begin{align}
\label{eqExponentialSize}
P(T_{b} <\infty)=\frac{1 - R_{b} - D_{b}}{b + O_{b}} = \frac{\theta_{1}}{b \theta_{0}}.
\end{align}
For \( (\theta_{0},\theta_{1})=(1,0.8) \) and \( \alpha=0.05 \), the standard boundary \( 20 \) has size \( 0.04 \), whereas the boundary \( 16 \) has size \( 0.05 \).

Under \( H_{1} \), Wald's identity and the same memorylessness property gives
\begin{align}
\label{eqExponentialTime}
\Ef_{Q}[T_{b} ]=\frac{\log b+(\theta_{0}-\theta_{1})/\theta_{1}}{\log(\theta_{1}/\theta_{0})+(\theta_{0}-\theta_{1})/\theta_{1}}.
\end{align}
%
The expected sample size falls from \( 120.8549 \) at \( b=20 \) to \( 112.5461 \) at \( b=16 \), a reduction of \( 6.9\% \).


With the true \( p_{b} \) at hand, we can evaluate the sampling procedure inspired by \refProp{propBracket} to derive \( b_{\alpha} \). \refFig{figExponential} shows that this results in the estimate \( \hat{b}_{\alpha} = 16.06 > b_{\alpha} \), thus, valid. Without \refEq{eqBracket} it is unclear how the naive \( N=500 \)-committed threshold estimate of \( \tilde{b}_{N, \alpha} = 15.77 \) should generalize to the horizon-free case. 
\begin{figure}[h]
\centering
\includegraphics[width=0.94 \textwidth]{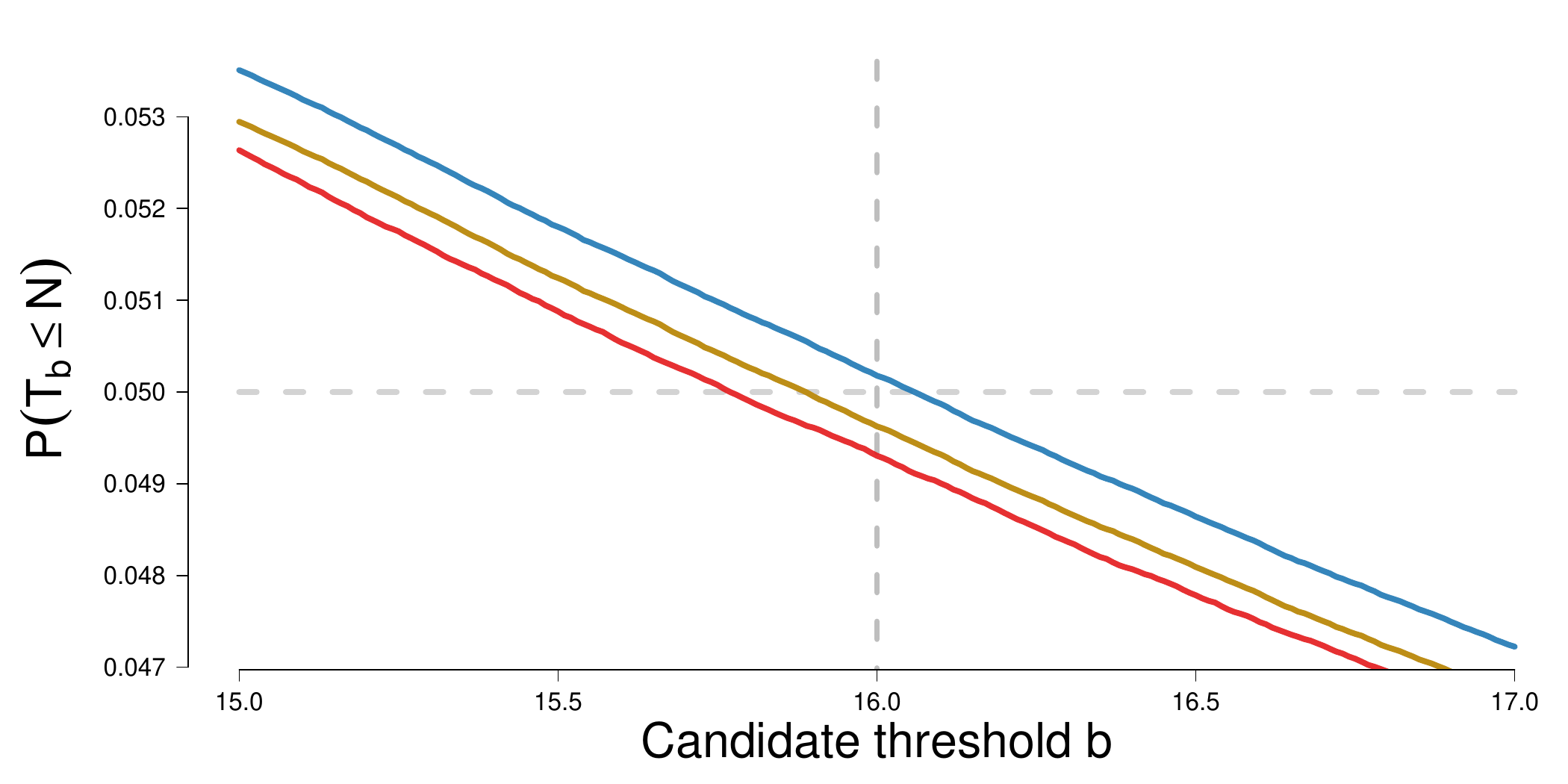}
\caption{By sampling up to a time horizon of \( N=500 \), we derive a valid estimate \( \hat{b}_{\alpha} = 16.06 \) of the horizon-free threshold \( b_{\alpha} = 16 \) (blue curve). The yellow and red curves represent estimates of the horizon committed Ville's identity estimator \( \hat{p}_{b, N} \) and the naive estimator \( \tilde{p}_{b,N} \), which cross \( \alpha =0.05 \) at thresholds \( \hat{b}_{N, \alpha} = 15.89 \) and \( \tilde{b}_{N, \alpha} = 15.77 \) respectively.} 
\label{figExponential}
\end{figure}

\subsection{A conformal test of exchangeability}
\label{secConformal}

The null models of the preceding examples are simple. We close with a composite nonparametric null. Let the data be \( z_{1}, z_{2}, \ldots \), such as \( z_{n} := (x_{n}, y_{n}) \) in a regression setting, and let \( A \) be a so-called inductive conformity measure that sends observations \( z \in Z \) to a real number such as \( a_{n} := A(z_{n}) := y_{n} - \bar{y}_{n} \) measuring the difference between the observed outcome \( y_{n} \) and a predicted outcome \( \bar{y}_{n} \) based on \( x_{n} \). The composite nonparametric null hypothesis states that the sequence of conformity scores \( a_{1}, a_{2}, \ldots \) is  exchangeable. Hence, the null excludes the possibility of a distribution shift %
%
%
\citep{vovk2003testing,vovk2021testing}. We map the sequence \( a_{1}, a_{2}, \ldots \) into a sequence of smoothed (conformal) \( p \)-value such that 
\begin{align}
\label{eqConfPVal}
p_{n} := \frac{\# \{ i \leq n : a_{i} < a_{n} \} + u_{n} \, \# \{ i  \leq n : a_{i} = a_{n} \}}{n} , 
\end{align}
where \( u_{1}, u_{2}, \ldots \) are i.i.d. uniform on \( [0,1] \) and independent \( z_{1}, z_{2}, \ldots \). The resulting smoothed \( p_{1}, p_{2}, \ldots \) are i.i.d. uniform on \( [0,1] \) under \( H_{0} \), regardless of the score \citep[Proposition~1]{vovk2021retrain}. These \( p \)-values are then transformed into \( e \)-values using \( p \)-to-\( e\) calibrators given by \( f_{\epsilon}(p) := 1 + \epsilon (p - \tfrac{1}{2}) \). Note that this requires \( \epsilon \in [-2,2] \), as then \( f_{\epsilon}(p) \geq 0 \) and \( \Ef [ f_{\epsilon}(P)] = \int_{0}^{1} f_{\epsilon} = 1 \) under the null. Taking the running product of the (independent) sequence of \( e \)-values then defines a nonnegative product martingale. A standard version of Vovk's simple jumper algorithm \citep[Algorithm~1]{vovk2021retrain} considers three bettors associated with \( \epsilon \in \{-1, 0, 1 \} \) leading to three accounts \( C^{\epsilon}_{n} \) with \( C^{-1}_{0} = 1/3 \), and \( M_{0} = C^{\epsilon}_{-1} + C^{0}_{0} + C^{1}_{0} = 1 \). At every step a fraction \( J = 0.01 \) of each account is shared equally amongst each other, and each account is then updated with its associated \( p \)-to-\( e \) calibrator: 
\begin{align}
\label{eqJumper}
\text{For } \epsilon \in \{ -1, 0, 1 \}: C^{\epsilon}_{n} := \Bigl[ (1-J) \, C^{\epsilon}_{n-1} + \frac{J}{3} \, M_{n-1} \Bigr] f_{\epsilon}(p_{n}), \qquad M_{n} := \sum_{\epsilon}
C^{\epsilon}_{n}, 
\end{align}
where \( M_{n} \) is the total wealth across the bettors, and note that \( M_{n} \) is a \( \Hc_{0} \) test martingale %
%
with respect to \( \Gc_{n}:= \sigma(p_{1}, \ldots, p_{n}) \). Theorem~\ref{thmFinite} applies with \( D_{b,N} = 0 \). The null is rejected as soon as \( M_{n} \geq b\).

Ville's inequality suggests the threshold test \( M_{n} \geq 1/\alpha \). To quantify the slack, we simulate \( m = 10^{6} \) sample paths and subsequently estimate \( p_{b,N} \), \( V_{b,N} \) and \( R_{b,N} \) for \( b = 20 \)  in addition to other candidate improved thresholds \( b \in (1, 1/\alpha] \). 
%
%
%
For \( b = 20 \) we have, at the horizons \( N = 200 \) and \( N = 2000 \),
%
%
%
%
%
%
%
%
%
\begin{align}
\label{eqConfDecomposition}
\begin{aligned}
N = 200 : \qquad 1 &\approx \underbrace{0.50617}_{20 p_{20,N}} + \underbrace{0.06497}_{V_{20,N}} + \underbrace{0}_{D_{20,N}} + \underbrace{0.42886}_{R_{20,N}}, \\
N = 2000 : \qquad 1 &\approx  \underbrace{0.88477}_{20 p_{20,N}} + \underbrace{0.11237}_{V_{20,N}} + \underbrace{0}_{D_{20,N}} + \underbrace{0.00286}_{R_{20,N}} . 
\end{aligned}
\end{align}
All terms except \( D_{20,N} = 0 \) are estimates. \( P(T_{20} \leq 200) \approx 0.02531 \) against the budget \( 0.05 \). The residual is substantial at \( N = 200 \) and negligible by \( N = 2000 \); dividing \( R_{b,N} \) by \( b \) yields the width of the bracket. 

Table~\ref{tabConf} displays the brackets and compares the universal Ville's inequality boundary \( 1/\alpha \) to the committed boundary \( b_{N,\alpha} = 11.1992 \), and the horizon-free boundary \( b_{\alpha} = 18.68956 \) derived using the sampling procedure for \( \alpha = 0.05 \) at \( N = 200\). 
%
%
%
%
At \( N = 200 \) the three boundaries are distinct and the committed threshold test has an estimated crossing probability of approximately
\(\alpha\) by \(N=200\), increasing beyond \(\alpha\)
under continued monitoring. 
%
The horizon-free threshold test rejects more often than the \( M_{n} \geq 20 \) test at every horizon, see 
Figure~\ref{figConf}.

%
%
%
%
%
%
%
%
%
%
\begin{figure}[ht]
\centering
\begin{tabular}{cc}
    \begin{minipage}{.48 \textwidth}
	    \centering
	    \includegraphics[width= \linewidth]{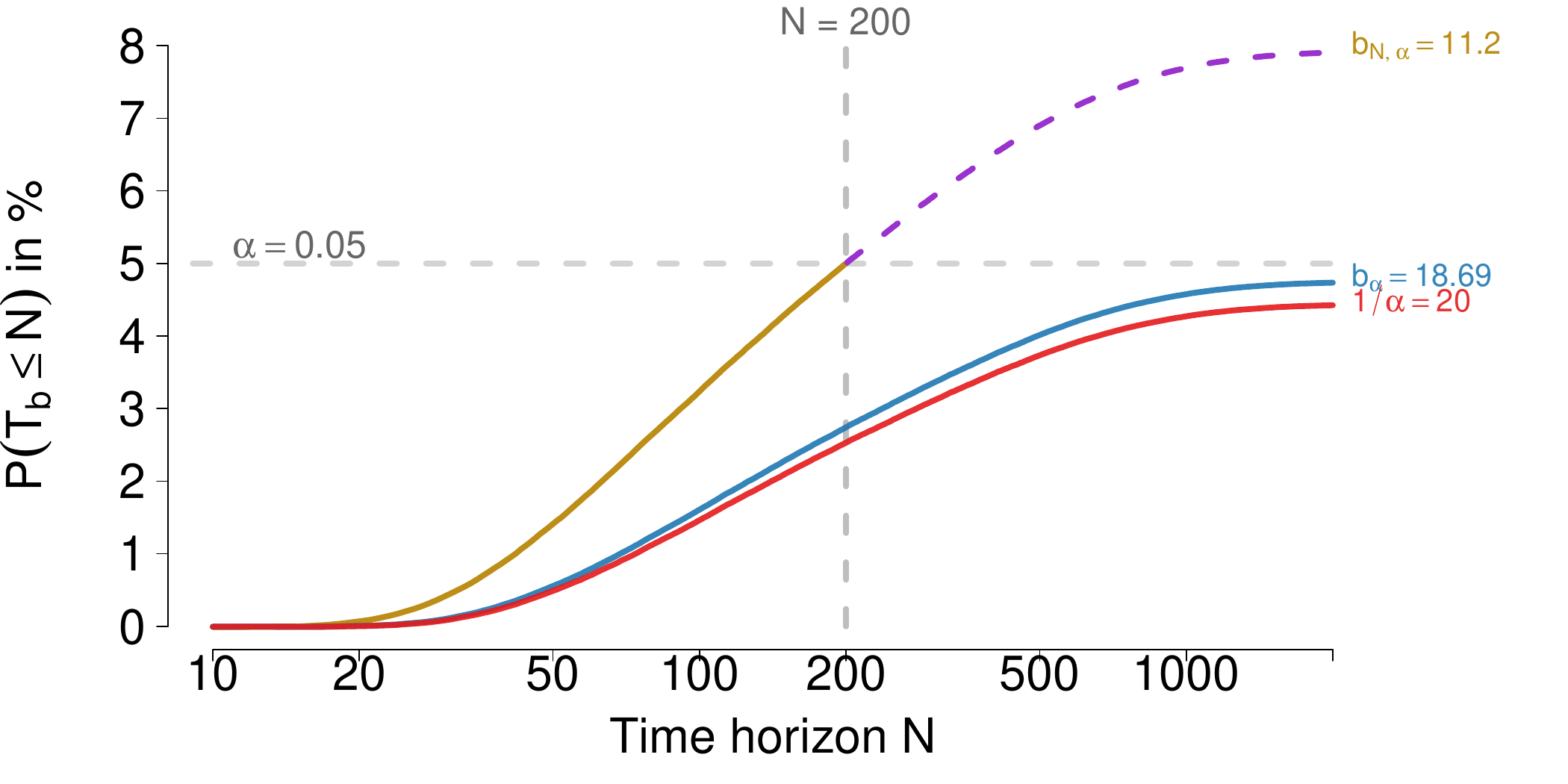}
    \end{minipage}  & %
    \begin{minipage}{.48 \textwidth}
	    \centering
	    \includegraphics[width=\linewidth]{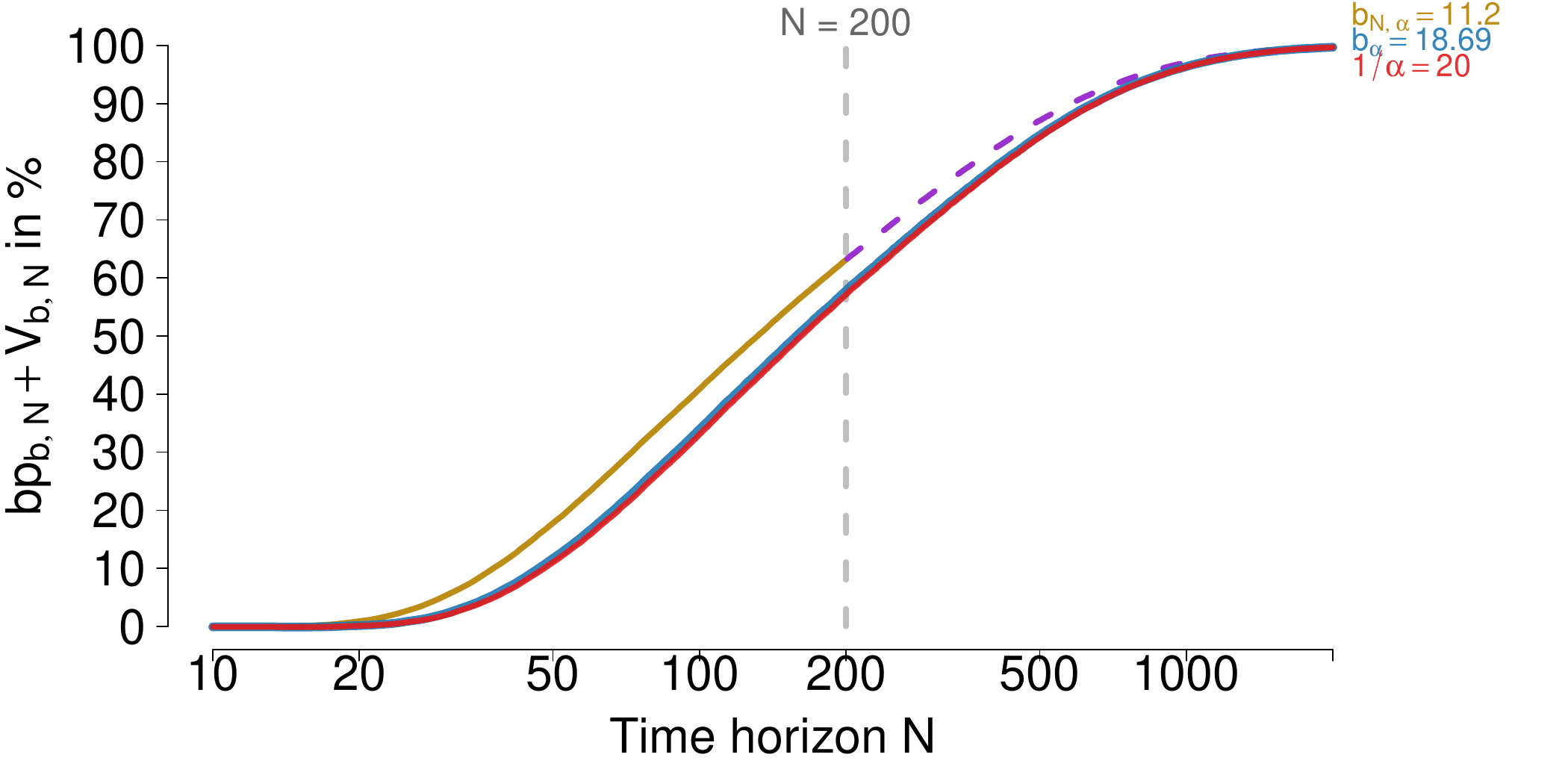}
    \end{minipage}  %
\end{tabular}
\caption{
The yellow, blue, and red curves represent the type I error and power curves of the procedures associated to the committed \( M_{n} \geq b_{N,\alpha}=11.1992 \), horizon-free \( M_{n} \geq b_{\alpha}=18.68956 \), and the universal \( M_{n} \geq 1/\alpha=20 \) test, respectively. The left panel shows that the committed test, yellow curve, is numerically calibrated through \(N=200\), with continued monitoring beyond this horizon depicted in purple. The estimated horizon-free test (blue) targets type I error control at every horizon, while the universal threshold test (red) guarantees it, but both have less ``power'', associated with some implied alternative under the null, as represented by \( b p_{b, N} + V_{b,n} \) at the horizon \( N = 200 \) (right panel).} %
\label{figConf}
\end{figure}

\begin{table}[h]
\centering
\caption{Ville's identity allows us to numerically calibrate lower rejection thresholds for \citeauthor{vovk2021retrain}'s \citeyear{vovk2021retrain} conformal test of exchangeability. The committed boundary \( b_{N,\alpha} \) and the horizon-free boundary \( b_{\alpha} \) are derived by sampling. Every probability is estimated at the reported \( N \) based on \( 10^{6}\) sample paths simulated under the null.} 
\label{tabConf}
\footnotesize
\setlength{\tabcolsep}{4pt}
\begin{tabular}{lllrcr@{}}
\hline
\( N \) & boundary & \( b \) & \( [\hat{p}_{b, N},  \hat{p}_{b,N} + \tfrac{\hat{R}_{b,N}}{b} ] \) & \( b \hat{p}_{b, N} + \hat{V}_{b,N}  \) \\ 
\hline
\( 200 \)  & committed \( b_{N,\alpha} \)  & \( 11.1992 \)  & \( [0.05000, \, 0.08294] \)   & \( 0.63115 \) \\ 
& horizon-free \( b_{\alpha} \) & \( 18.6896 \)  & \( [0.02743, \, 0.05000] \)   & \( 0.57813 \) \\ 
& universal \( 1/\alpha \)      & \( 20 \)      & \( [0.02531, \, 0.04675] \)    & \( 0.57114 \) \\ 
\( 2000 \) & committed \( b_{N,\alpha} \)  & \( 17.6908 \)  & \( [0.05000, \, 0.05015] \)   & \( 0.99730 \) \\ 
& horizon-free \( b_{\alpha} \) & \( 17.7504 \)  & \( [0.04985, \, 0.05000] \)   & \( 0.99729 \) \\ 
& universal \( 1/\alpha \)      & \( 20 \)      & \( [0.04424, \, 0.04438] \)     & \( 0.99714 \) \\ 
\hline
\end{tabular}
\end{table}
%
%
%

\section{Continuous-time extension}
\label{secContinuousTime}

%
%
%
We state Ville's identity for continuous-time supermartingales. The continuous-time identity has the same interpretation as its discrete-time counterpart. Let \( (\Omega, \Fc, (\Fc_{t})_{ t \geq 0}, \Pf) \) satisfy the usual conditions: the filtration is \( \Pf \)-complete (\( \Fc_{0} \) contains all \( \Pf \)-null sets) and is right-continuous. Let \( M \) be a nonnegative c\`{a}dl\`{a}g supermartingale, i.e., having right-continuous sample paths with left limits that starts at one (in particular, jumps are allowed); \( M_{0} = 1 \). Note that Ville's inequality would be false without the c\`{a}dl\`{a}g assumption. For \( b >1 \), set 
\begin{align}
T_{b} = \inf \{ t \geq 0: M_{t} \geq b \}, \quad \inf \emptyset := \infty
\end{align}
Right-continuity gives \( M_{T_{b}} \geq b \) on \( \{ T_{b} < \infty \} \). The continuous-time form of Ville's inequality is
\begin{align}
\Pf(T_{b} <\infty) = \Pf (\exists t \geq 0: M_{t} \geq b) \leq \frac{1}{b}.
\end{align}
A process is of class \( D \) if it forms a uniformly integrable family when the process is restricted to finite stopping times. 

\begin{lemma}
\label{lemmaClassD}
The stopped process \( X_{t} = M_{T_{b} \wedge t} \) is a supermartingale of class \( D \). In particular, it has a Doob--Meyer decomposition %
\begin{align}
\label{eqDoobMeyer}
 X_{t}=1+L^{(b)}_{t}-A^{(b)}_{t},
\end{align}
where \( L^{(b)}_{0}=A^{(b)}_{0}=0 \), \( L^{(b)} \) is a uniformly integrable martingale, and \( A^{(b)} \) is predictable, c\`{a}dl\`{a}g, and increasing, with \( \Ef [A^{(b)}_\infty ] \leq 1 \). \( \hfill \diamond \)
\end{lemma}

\begin{proof}
%
%
We begin with a finite time interval of length \( n \) and apply Doob's optional stopping theorem at the finite stopping time \( T_{b} \wedge n \), which gives \( \mathbb E [ M_{T_{b} \wedge n}] \leq \mathbb E[M_0] = 1 \). Since \( M_{T_{b} \wedge n} \geq 0 \) and equals \( M_{T_{b}} \) eventually on \( \{ T_{b} < \infty \} \), we have pathwise that
\begin{align}
    M_{T_b} \ind{T_b <\infty} \leq \liminf_{n\to\infty} M_{T_b \wedge n}.
\end{align}

By Fatou's lemma,

\begin{align}
 \mathbb E \bigl [M_{T_{b}} \ind{T_{b}<\infty} \bigr]&\leq \mathbb E\left[\liminf_{n\to\infty} M_{T_b\wedge n}\right]\\
 &\le \liminf_{n\to\infty} \mathbb E[M_{T_b\wedge n}] \leq 1
\end{align}

Every finite stopping time \( \sigma \) satisfies the pathwise bound

\begin{align}
 0 \leq X_{\sigma} \leq b+M_{T_{b}} \ind{T_{b}<\infty}.
\end{align}

This is because there are two possibilities for any finite stopping time $\sigma$: if $\sigma <T_b$, then $X_{\sigma} = M_{\sigma} < b$. If $T_b \leq \sigma$, then $X_{\sigma} = M_{T_n}$.

The right side is a single integrable random variable, proving class \( D \). Apply the Doob--Meyer theorem to \( X \) on each finite interval $[0,n]$, obtaining
\(X=1+L^{[n]}-A^{[n]}\), with
\(L^{[n]}_0=A^{[n]}_0=0\); the decomposition is unique up to indistinguishability \citep{beiglboeck2012short}. In particular, restricting the
decomposition on $[0,n+1]$ to $[0,n]$ gives another decomposition of $X$ on $[0,n]$, so uniqueness implies that the two
decompositions agree throughout $[0,n]$, almost surely. There are only countably many such intervals; hence, outside a single null event, all these decompositions agree wherever their intervals overlap. We can therefore define $L^{(b)}_t=L^{[n]}_t$
and $A^{(b)}_t=A^{[n]}_t$ using any integer $n\ge t$: every such choice of $n$ gives the same values. This yields a decomposition
\(X=1+L^{(b)}-A^{(b)}\) on all of \([0,\infty)\).

Since
\begin{align}
 \Ef [ A^{(b)}_{t}] =1- \Ef [ X_{t} ] \leq 1,
\end{align}
monotone convergence applies and yields \( \Ef[ A^{(b)}_{\infty}] \leq 1 \). Moreover, at all finite stopping times the martingale \( 1+L^{(b)}=X+A^{(b)} \) is dominated by the integrable random variable \( b+M_{T_{b}} \ind{T_{b} < \infty} + A^{(b)}_{\infty} \). It is therefore uniformly integrable on the entire time axis.
\end{proof}

For \( t \geq 0 \), define
\begin{align}
 \label{eqContTerms1}
 p_{b,t} & :=\Pf(T_{b} \leq t),& V_{b,t}& :=\E[(M_{T_{b}}-b)\ind{T_{b} \leq t}],\\
  \label{eqContTerms2}
 D_{b,t} & :=\E [ A^{(b)}_{t} ], &  R_{b,t}& :=\E[M_{t} \ind{T_{b}>t}].
\end{align}
\begin{theorem}[Continuous-time identity]
\label{thmContVilleIdentity}
For every finite \( t \geq 0 \),
\begin{align}
\label{eqContFinite}
1=b p_{b,t}+V_{b,t}+D_{b,t}+R_{b,t}.
\end{align}
All four terms are finite and nonnegative. The functions \( t \mapsto p_{b,t},V_{b,t},D_{b,t} \) are nondecreasing and \( t \mapsto R_{b,t} \) is nonincreasing. If \( M_{\infty}= \lim_{t \rightarrow \infty} M_{t} \), %
%
%
then
\begin{align}
\label{eqContInfinite}
 1=b \Pf(T_{b}<\infty)+V_{b}+D_{b}+R_{b}, \qquad R_{b}=\Ef[M_{\infty} \ind{T_{b}=\infty}],
\end{align}
where \( V_{b}=\lim_{t} V_{b,t} \) and \( D_{b}=\Ef [ A^{(b)}_{\infty}] \). Consequently,
\begin{align}
\label{eqContBracket}
 p_{b,t} \le \Pf(T_{b}<\infty) \leq p_{b,t} + \frac{R_{b,t}-R_{b}}{b} \leq p_{b,t}+\frac{R_{b,t}}{b}.
\end{align}
The two upper bounds above are nonincreasing in \( t \). \( \hfill \diamond \)
\end{theorem}

\begin{proof}
Taking expectations in \refEq{eqDoobMeyer} gives \( \Ef[ M_{T_{b} \wedge t}]=1-D_{b,t} \). Splitting according to \( \{ T_{b} \leq t \} \) gives
\begin{align}
\Ef[  M_{T_{b} \wedge t}] =b p_{b,t}+V_{b,t}+R_{b,t},
\end{align}
which proves \refEq{eqContFinite}. $A^{(b)}_t$ increases pathwise, so $D_{b,t} = \mathbb E[A^{(b)}_t]$ increases. For the overshoot term, the random variable $(M_{T_b} -b)\ind{T_b\le t}$ increases pathwise, so its expectation $V_{b,t}$ increases. $p_{b,t}$ increases because as $t$ increases, more paths may cross. Therefore, $R_{b,t}$ decreases. The supermartingale convergence theorem gives a finite almost-sure limit \( M_{\infty} \). Since
\begin{align}
 0 \leq M_{t} \ind{T_{b}>t} < b,\qquad M_{t} \ind{T_{b} > t } \longrightarrow M_{\infty} \ind{T_{b}=\infty},
\end{align}
dominated convergence yields the limit \( R_{b} \). Monotone convergence theorem handles the other terms. Subtracting \refEq{eqContFinite} from \refEq{eqContInfinite} gives
\begin{align}
 R_{b,t}-R_{b} =b\Pf(t<T_{b}<\infty)+(V_{b}-V_{b,t})+(D_{b}-D_{b,t}),
\end{align}
which proves the bounds. Their monotonicity also follows from \( p_{b,t}+R_{b,t}/b=(1-V_{b,t}-D_{b,t})/b \).
\end{proof}

\begin{corollary}[Continuous crossing and the classical maximal identity]
\label{corNoOvershoot}
If \( M \) has no positive jumps, then \( V_{b,t}=V_{b}=0 \). If \( M \) is a nonnegative local martingale, then \( D_{b,t}=D_{b}=0 \), whether or not it has positive jumps. If both conditions hold, then
\begin{align}
\label{ct:maximal}
 \Pf( T_{b} < \infty)=\frac{1-\E[M_{\infty} \ind{T_{b}=\infty}]}{b}.
\end{align}
In particular, if \( M_{\infty} =0 \) almost surely, then \( \Pf(T_{b}<\infty)=1/b \). \( \hfill \diamond \)
\end{corollary}

\begin{proof}
Before crossing, the process is below \( b \); absence of positive jumps therefore forces \( M_{T_{b}}=b \) whenever \( T_{b}<\infty \). If \( M \) is a local martingale, its stopped process is a local martingale of class \( D \) by Lemma~\ref{lemmaClassD}, hence a uniformly integrable martingale. Since the process is a martingale, its predictable drift term is zero and hence \( D_{b} = 0 \). When positive jumps are absent, this can also be seen directly from \( 0 \leq M_{T_{b} \wedge t} \leq b \). Apply \refEq{eqContInfinite}.
\end{proof}

The last statement recovers the classical Doob maximal identity. \citet[Definition~1.1 and Lemma~2.1]{nikeghbali2006doob} state its unconditional and conditional forms for strictly positive local martingales with no positive jumps and terminal limit zero. 


\begin{example}[Brownian exponential]
\label{exBrownian}
We calculate the correction terms to illustrate continuous time Ville's identity. Let \( B \) be standard Brownian motion, \( \sigma > 0 \), \( c \geq 0 \), and
\begin{align}
 M_{t} =\exp\{\sigma B_t-(\sigma^2/2+c)t\}.
\end{align}
It\^{o}'s formula gives
\begin{align}
 dM_{t} =\underbrace{\sigma M_{t} \,\der B_t}_{\text{martingale part}}-\underbrace{cM_{t} \,\der t}_{\text{negative drift}}.
\end{align}
Here we have $M_0 = 1$ and there is no overshoot (since the sample paths are continuous). Stopping and integrating gives
\begin{equation}
    M_{T_b \wedge t} = 1+\sigma \int_0^t \ind{s\le T_{b}} M_s \der B_s - c\int_0^{T_b \wedge t} M_s \der s.
\end{equation}
Therefore, we have 
\begin{equation}
    \mathbb E\left[M_{T_b \wedge t}\right] = 1- c\ \mathbb E \left[\int_0^{T_b\wedge t} M_s \der s\right].
\end{equation}
Since $M_{T_b} = b$, we then have
\begin{align}
\label{eqBrownianFinite}
1=b\Pf(T_b\leq t)
+c\Ef\left[\int_0^{T_b\wedge t}M_s\,\der s\right]
+\Ef[M_t\ind{T_b>t}].
\end{align}

Brownian motion satisfies \( B_{t} / t \rightarrow 0 \) almost surely, so \( M_{t} \rightarrow 0 \) and $R_b = 0$. Here the crossing probability can be calculated directly, as we next show, so Ville's identity does not have to do the calibration work in this example. Instead, it tells us where the unused testing level went: entirely into predictable decrease. To calculate the crossing probability, put \( \kappa=1+2c/\sigma^{2} \). Then
\begin{align}
 M_{t} ^{\kappa}=\exp \{ \kappa \sigma B_{t}-(\kappa\sigma)^{2} t/2 \}
\end{align}
is an exponential martingale. Stopping it at \( T_{b} \), it is bounded by \( b^{\kappa} \) and converges to \( b^{\kappa}\ind{T_{b}<\infty} \). Therefore
\begin{align}
\label{eqBrownianExact}
 p_b=b^{-\kappa},\qquad V_{b}=R_{b}=0,\qquad D_{b}=1-b^{-2c/\sigma^2}.
\end{align}
For \( c=0 \), the usual threshold \( b=1/\alpha \) uses the entire level \( \alpha \). For \( c>0 \), all horizon-free slack is predictable decrease; the exact level-\( \alpha \) boundary is \( b=\alpha^{-1/\kappa}<1/\alpha \). For example, \( \sigma^{2}=1 \) and \( c=1/2 \) give \( \kappa=2 \): at level \( 0.05 \) the exact boundary is \( \sqrt{20} \), while the standard boundary \( 20 \) has crossing probability \( 1/400 \).
\end{example}

\section{Discussion}
\label{secDiscussion}

The finite-horizon Ville's identity 
%
separates the causes that make Ville's inequality loose. A path can cross above the boundary rather than on it; the process can lose conditional expectation before crossing; or its value can remain on paths that have not crossed. The same three terms persist even without a horizon, where the last term becomes the expected terminal value on noncrossing paths.

The decomposition also suggests different responses. Large overshoot \( V_{b} \) points toward a different boundary, thus, a revised threshold test with a less stringent boundary. A large predictable-decrease term \( D_{b}  \) indicates that the chosen statistic is a strict supermartingale under the null. Replacing the statistic by a test martingale then removes this term. A large finite-horizon residual \( R_{b,N} \) may suggest a small effect size or a time horizon that is set too low. Extending the horizon typically reduces the residual, while lowering the boundary can reduce it more quickly at the cost of losing type I error control beyond a fixed time horizon \( N \). 

We also showed how to exploit Ville's identity (\refEq{eqBracket}) to lower the boundary such that the threshold test \( M_{n} \geq b_{\alpha} \) remains valid under indefinite monitoring. The rough width is \( R_{b,N} / b \): %
%
%
what a finite computation leaves undecided about indefinite monitoring is exactly the value still carried by noncrossing paths. When \(R_b\) is known, we can use the finer width \((R_{b,N}-R_b)/b\). A constant shift of the process instead requires a deterministic pathwise lower bound, as in \refSec{secLowerBound}.

The available boundaries are therefore ordered. The universal boundary \( 1/\alpha \) requires neither computation nor a commitment to a time horizon. The horizon-free boundary \( b_{\alpha} \) requires one finite computation and remains valid at every horizon. The committed boundary \( b_{N, \alpha} \) requires both, and rejects most often, but is valid only up to the fixed horizon \( N \). Table~\ref{tabConf} showcases all three in the conformal example.

For likelihood ratio martingales, the change-of-measure result makes the interpretation sharper: the residual is the type II error under the alternative. In the Bernoulli example, the rejection probability under the alternative is recovered directly from the null calculation as \( bp+V \). This provides a direct answer to the concern that a test should not be judged only by how much type I error it uses. The relevant quantity is the power obtained under the intended design.



The numerical examples distinguish exact calibration from the sampling procedure. In the Bernoulli example, the identity-based estimate selects a valid boundary where the empirical crossing frequency selects an invalid one. The conformal example extends the numerical illustration to a common null law shared by a composite family. Its estimated thresholds remain subject to simulation error.
Characterizing the accuracy of these calibration procedures, and incorporating simulation uncertainty into their guarantees,
are natural and practical further work.



\section*{Acknowledgements}
Victor H. de la Pe\~na acknowledges support from Google DeepMind and would like to thank Rianne de Heide and Wouter Koolen for the opportunity to present a first version of this work at the fourth SAVI workshop held in Enschede, the Netherlands, June 2026. The contribution of Alexander Ly to this publication is part of the project “Increasing Scientific Efficiency with Sequential Methods” with file number VI.Veni.211G.040 of the NWO Talent Programme Veni, which is financed by the Dutch Research Council (NWO). Fangyuan Lin thanks Lily Xu for introducing him to practical auditing problems, where repeated rounds of methodological refinement made clear how statistical conservatism translates into costly expert effort.

\bibliographystyle{plainnat}
\bibliography{refs}

@article{blier-wong-wang2026,
  author  = {Blier-Wong, Christopher and Wang, Ruodu},
  title   = {Improved thresholds for e-values},
  journal = {The Annals of Statistics},
  year    = {forthcoming},
  note    = {Accepted for publication}
}

@article{choe2026combining,
  title={Combining evidence across filtrations},
  author={Choe, Yo Joong and Ramdas, Aaditya},
  journal={Journal of the Royal Statistical Society Series B: Statistical Methodology},
  pages={qkag058},
  year={2026},
  publisher={Oxford University Press UK}
}

@article{delapena-lin-delapena2026,
  author        = {de la Pe\~na, Victor H. and Lin, Fangyuan and de la Pe\~na, Victor K.},
  title         = {A selection premium decomposition for the expected maximum of random walks},
  journal       = {arXiv preprint},
  year          = {2026},
  eprint        = {2602.19481},
  archivePrefix = {arXiv}
}

@book{doob1953stochastic,
  title={Stochastic Processes},
  author={Doob, J.L.},
  isbn={9780471218135},
  lccn={lc52011857},
  series={Probability and Statistics Series},
  url={https://books.google.com/books?id=KvJQAAAAMAAJ},
  year={1953},
  publisher={Wiley}
}

@article{fischer2026improving,
  title={Improving Wald’s (approximate) sequential probability ratio test by avoiding overshoot},
  author={Fischer, Lasse and Ramdas, Aaditya},
  journal={IEEE Transactions on Information Theory},
  year={2026},
  publisher={IEEE}
}

@article{grunwald2024safe,
  title={Safe testing},
  author={Gr{\"u}nwald, Peter and de Heide, Rianne and Koolen, Wouter},
  journal={Journal of the Royal Statistical Society Series B: Statistical Methodology},
  volume={86},
  number={5},
  pages={1091--1128},
  year={2024},
  publisher={Oxford University Press UK}
}

@article{howard2020time,
  author  = {Howard, Steven R. and Ramdas, Aaditya and McAuliffe, Jon and Sekhon, Jasjeet},
  title   = {Time-uniform {C}hernoff bounds via nonnegative supermartingales},
  journal = {Probability Surveys},
  volume  = {17},
  pages   = {257--317},
  year    = {2020},
  doi     = {10.1214/18-PS321}
}

@article{howard2021time,
  title={Time-uniform, nonparametric, nonasymptotic confidence sequences},
  author={Howard, Steven R and Ramdas, Aaditya and McAuliffe, Jon and Sekhon, Jasjeet},
  journal={The Annals of Statistics},
  volume={49},
  number={2},
  pages={1055--1080},
  year={2021},
  publisher={JSTOR}
}

@article{kim1987design,
  title={Design and analysis of group sequential tests based on the type I error spending rate function},
  author={Kim, Kyungmann and Demets, David L},
  journal={Biometrika},
  volume={74},
  number={1},
  pages={149--154},
  year={1987},
  publisher={Oxford University Press}
}

@article{klass-delapena2026,
  author        = {Klass, Michael J. and de la Pe\~na, Victor H.},
  title         = {Beyond {W}ald's equation and the optional sampling theorem},
  journal       = {arXiv preprint},
  year          = {2026},
  eprint        = {2601.17175},
  archivePrefix = {arXiv}
}

@article{koolen2025generalisation,
  title={A generalisation of ville’s inequality to monotonic lower bounds and thresholds},
  author={Koolen, Wouter M and P{\'e}rez-Ortiz, Muriel F and Lardy, Tyron},
  journal={Statistics \& Probability Letters},
  pages={110577},
  year={2025},
  publisher={Elsevier}
}

@article{lan1983discrete,
  title={Discrete sequential boundaries for clinical trials},
  author={Lan, KK Gordon and DeMets, David L},
  journal={Biometrika},
  pages={659--663},
  year={1983},
  publisher={JSTOR}
}

@book{lehmann-romano2005,
  author    = {Lehmann, Erich L. and Romano, Joseph P.},
  title     = {Testing Statistical Hypotheses},
  edition   = {3},
  publisher = {Springer},
  year      = {2005},
  doi       = {10.1007/0-387-27605-X}
}

@article{lorden1970excess,
  title={On excess over the boundary},
  author={Lorden, Gary},
  journal={The Annals of Mathematical Statistics},
  volume={41},
  number={2},
  pages={520--527},
  year={1970},
  publisher={Institute of Mathematical Statistics}
}

@article{nikeghbali-yor2006,
  author  = {Nikeghbali, Ashkan and Yor, Marc},
  title   = {Doob's maximal identity, multiplicative decompositions and enlargements of filtrations},
  journal = {Illinois Journal of Mathematics},
  volume  = {50},
  number  = {1--4},
  pages   = {791--814},
  year    = {2006}
}

@article{ramdas2026randomized,
  title={Randomized and exchangeable improvements of Markov’s, Chebyshev’s and Chernoff’s inequalities},
  author={Ramdas, Aaditya and Manole, Tudor},
  journal={Statistical Science},
  volume={41},
  number={1},
  pages={121--142},
  year={2026},
  publisher={Institute of Mathematical Statistics}
}

@article{ramdas2023game,
  title={Game-theoretic statistics and safe anytime-valid inference},
  author={Ramdas, Aaditya and Gr{\"u}nwald, Peter and Vovk, Vladimir and Shafer, Glenn},
  journal={Statistical Science},
  volume={38},
  number={4},
  pages={576--601},
  year={2023},
  publisher={Institute of Mathematical Statistics}
}

@book{siegmund2013sequential,
  title={Sequential analysis: tests and confidence intervals},
  author={Siegmund, David},
  year={2013},
  publisher={Springer Science \& Business Media}
}

@article{terschure-grunwald2025,
  author  = {ter Schure, Judith and Gr{\"u}nwald, Peter D.},
  title   = {{ALL-IN} meta-analysis: Breathing life into living systematic reviews and prospective meta-analyses},
  journal = {F1000Research},
  volume  = {11},
  pages   = {549},
  year    = {2025},
  doi     = {10.12688/f1000research.74223.2}
}

@book{ville1939etude,
  title={Etude critique de la notion de collectif},
  author={Ville, Jean},
  volume={3},
  year={1939},
  publisher={Gauthier-Villars Paris}
}

@article{vovk2021values,
  title={E-values: Calibration, combination and applications},
  author={Vovk, Vladimir and Wang, Ruodu},
  journal={The Annals of Statistics},
  volume={49},
  number={3},
  pages={1736--1754},
  year={2021},
  publisher={Institute of Mathematical Statistics}
}

@book{wald2004sequential,
  title={Sequential analysis},
  author={Wald, Abraham},
  year={2004},
  publisher={Courier Corporation}
}

@article{wang2025extended,
  title={The extended Ville’s inequality for nonintegrable nonnegative supermartingales},
  author={Wang, Hongjian and Ramdas, Aaditya},
  journal={Bernoulli},
  volume={31},
  number={4},
  pages={2723--2746},
  year={2025},
  publisher={Bernoulli Society for Mathematical Statistics and Probability}
}

@article{wang2025,
  author  = {Wang, Ruodu},
  title   = {The only admissible way of merging arbitrary e-values},
  journal = {Biometrika},
  volume  = {112},
  number  = {2},
  pages   = {asaf020},
  year    = {2025},
  doi     = {10.1093/biomet/asaf020}
}

@book{woodroofe1982nonlinear,
  title={Nonlinear renewal theory in sequential analysis},
  author={Woodroofe, Michael},
  year={1982},
  publisher={SIAM}
}

@inproceedings{vovk2003testing,
  title={Testing exchangeability on-line},
  author={Vovk, Vladimir and Nouretdinov, Ilia and Gammerman, Alexander},
  booktitle={Proceedings of the 20th international conference on machine learning (ICML-03)},
  pages={768--775},
  year={2003}
}

@inproceedings{vovk2021retrain,
  title={Retrain or not retrain: Conformal test martingales for change-point detection},
  author={Vovk, Vladimir and Petej, Ivan and Nouretdinov, Ilia and Ahlberg, Ernst and Carlsson, Lars and Gammerman, Alex},
  booktitle={Conformal and Probabilistic Prediction and Applications},
  pages={191--210},
  year={2021},
  organization={PMLR}
}

@article{vovk2021testing,
  author  = {Vovk, Vladimir},
  title   = {Testing randomness online},
  journal = {Statistical Science},
  volume  = {36},
  number  = {4},
  pages   = {595--611},
  year    = {2021}
}

@article{ly2025tutorial,
  title={A tutorial on safe anytime-valid inference: Practical maximally flexible sampling designs for experiments based on e-values},
  author={Ly, Alexander and Boehm, Udo and Gr{\"u}nwald, Peter and Ramdas, Aaditya and van Ravenzwaaij, Don},
  year={2025},
  journal={PsyArxiv preprint},
  doi={10.31234/osf.io/h5vae},
  publisher={OSF}
}

@article{ramdas2020admissible,
  title={Admissible anytime-valid sequential inference must rely on nonnegative martingales},
  author={Ramdas, Aaditya and Ruf, Johannes and Larsson, Martin and Koolen, Wouter},
  journal={arXiv preprint arXiv:2009.03167},
  year={2020}
}

@article{ruf2023composite,
  title={A composite generalization of Ville’s martingale theorem using e-processes},
  author={Ruf, Johannes and Larsson, Martin and Koolen, Wouter M and Ramdas, Aaditya},
  journal={Electronic Journal of Probability},
  volume={28},
  pages={1--21},
  year={2023},
  publisher={The Institute of Mathematical Statistics and the Bernoulli Society}
}

@article{chugg2026values,
  title={E-values as statistical evidence: A comparison to {B}ayes factors, likelihoods, and p-values},
  author={Chugg, Ben and Ramdas, Aaditya and Gr{\"u}nwald, Peter},
  journal={Synthese},
  volume={208},
  number={3},
  pages={130},
  year={2026},
  publisher={Springer}
}

@article{ram2026power,
  title={Power one sequential tests exist for weakly compact $\mathcal{P}$ against $\mathcal{P}^{c}$},
  author={Ram, Ashwin and Ramdas, Aaditya},
  journal={arXiv preprint arXiv:2604.03218},
  year={2026}
}

@article{baas2026adaptive,
  title={Adaptive clinical trials based on design-optimal e-values with automatic curtailment: An application to single-arm trials with binary data},
  author={Baas, Stef and ter Schure, Judith and van Rosmalen, Joost},
  journal={arXiv preprint arXiv:2605.28653},
  year={2026}
}

@article{clerico2026time,
  title={Time-sensitive anytime-valid testing},
  author={Clerico, Eugenio and Wegel, Tobias and Azangulov, Iskander and Rebeschini, Patrick},
  journal={arXiv preprint arXiv:2605.06521},
  year={2026}
}

@book{wassmer2025group,
  title={Group Sequential and Confirmatory Adaptive Designs in Clinical Trials},
  author={Wassmer, Gernot and Brannath, Werner},
  year={2025},
  publisher={Springer Nature}
}

@article{beiglboeck2012short,
  title={A short proof of the Doob--Meyer theorem},
  author={Beiglboeck, Mathias and Schachermayer, Walter and Veliyev, Bezirgen},
  journal={Stochastic Processes and their applications},
  volume={122},
  number={4},
  pages={1204--1209},
  year={2012},
  publisher={Elsevier}
}

@article{nikeghbali2006doob,
  title={Doob's maximal identity, multiplicative decompositions and enlargements of filtrations},
  author={Nikeghbali, Ashkan and Yor, Marc},
  journal={Illinois Journal of Mathematics},
  volume={50},
  number={1-4},
  pages={791--814},
  year={2006},
  publisher={Duke University Press}
}

@article{lin2026not,
  title={Not All Mislabels Are Equal: Efficient Task-Aware Auditing with Confidence Sequences},
  author={Lin, Fangyuan and Krever, Noah and {\v{C}}ern{\`y}, Jakub and Xu, Lily},
  journal={arXiv preprint arXiv:2610.33262},
  year={2026}
}

@article{brannath2026anytime,
  title={Anytime-valid testing with e-values and confirmatory adaptive designs},
  author={Brannath, Werner and Fischer, Lasse},
  journal={arXiv preprint arXiv:2606.00878},
  year={2026}
}

@article{de2026mitigating,
  title={Mitigating the Winner's Curse While Controlling Multiplicity: e-Process Methods for Anytime-Valid Inference in Dose-Ranging Trials},
  author={de la Pena, Victor K and Lin, Fangyuan and Alemayehu, Demissie and de la Pena, Victor H},
  journal={arXiv preprint arXiv:2607.07699},
  year={2026}
}

@article{ding2026reference,
  title={Reference-Null Calibrated Thresholds for E-Processes with Applications to Conformal Martingales},
  author={Ding, Yi and Wei, Lan and Zhu, Xuehu and Dai, Wenlin},
  journal={arXiv preprint arXiv:2609.32678},
  year={2026}
}

\end{document}